\documentclass[amscd,amssymb,verbatim,10pt]{amsart}
\usepackage[fontsize = 10bp]{fontsize}
\usepackage{scalerel}
\usepackage{stackengine,wasysym}
\usepackage{bbm}
\usepackage{enumerate}
\usepackage{mathtools}

\usepackage{amssymb,latexsym, amsmath, amscd, array, graphicx, pb-diagram}
\usepackage{hyperref}
\usepackage{color}

\usepackage{tikz-cd}

\usepackage{color}
\hypersetup{
    colorlinks=true,
    linkcolor=red,
    urlcolor=blue,
    citecolor=blue
           }

\usepackage{amsthm}
\usepackage{hyperref}
\usepackage[all]{xy}
\usepackage[scr]{rsfso}

\numberwithin{equation}{section}

\swapnumbers

\theoremstyle{plain}

\newtheorem{theorem}{Theorem}[section]

\newtheorem{lemma}[theorem]{Lemma}

\newtheorem{proposition}[theorem]{Proposition}
\newtheorem{corollary}[theorem]{Corollary}

\theoremstyle{definition}

\newtheorem{definition}[theorem]{Definition}

\newtheorem{remark}[theorem]{Remark}

\newtheorem{ex}[theorem]{Example}

 \newcommand{\Wi}{\widetilde}

\DeclareMathOperator{\cat}{{\mathsf{cat}}}
\DeclareMathOperator{\TC}{{\mathsf{TC}}}

\DeclareMathOperator{\secat}{{\mathsf{secat}}}
\DeclareMathOperator{\dsecat}{{\mathsf{dsecat}}}
\DeclareMathOperator{\supp}{{\rm supp}}

\DeclareMathOperator{\cu}{{\mathsf{cup}}}
\DeclareMathOperator{\dcat}{{\mathsf{dcat}}}

\DeclareMathOperator{\dTC}{{\mathsf{dTC}}}

\DeclareMathOperator{\cd}{{\mathsf{cd}}}

\DeclareMathOperator{\Ker}{{\rm Ker}}

\def\id{\protect\operatorname{Id}}
\def\Im{\protect\operatorname{Im}}

\def\pr{\protect\operatorname{pr}}

\newcommand{\cA}{{\mathcal{A}}}
\newcommand{\cB}{{\mathcal{B}}}
\newcommand{\cG}{{\mathcal{G}}}

\makeatletter
\def\@tocline#1#2#3#4#5#6#7{\relax
  \ifnum #1>\c@tocdepth 
  \else
    \par \addpenalty\@secpenalty\addvspace{#2}%
    \begingroup \hyphenpenalty\@M
    \@ifempty{#4}{%
      \@tempdima\csname r@tocindent\number#1\endcsname\relax
    }{%
      \@tempdima#4\relax
    }%
    \parindent\z@ \leftskip#3\relax \advance\leftskip\@tempdima\relax
    \rightskip\@pnumwidth plus4em \parfillskip-\@pnumwidth
    #5\leavevmode\hskip-\@tempdima
      \ifcase #1
       \or\or \hskip 1em \or \hskip 2em \else \hskip 3em \fi%
      #6\nobreak\relax
    \hfill\hbox to\@pnumwidth{\@tocpagenum{#7}}\par
    \nobreak
    \endgroup
  \fi}
\makeatother

\def\scr{\mathcal}
\def\A{{\scr A}}
\def\B{{\scr B}}

\def\C{{\mathbb C}}

\def\Q{{\mathbb Q}}
\def\R{{\mathbb R}}

\newcommand{\D}{{\sf {D}}}
\newcommand{\dep}{{\sf {dep}}}
\newcommand{\sd}{{(\sf {d})}}
\newcommand{\dD}{{\sf {dD}}}

\newcommand{\cD}{{\mathcal{D}}}

\def\1{\hbox{\rm\rlap {1}\hskip.03in{\rom I}}}
\def\Bbbone{{\rm1\mathchoice{\kern-0.25em}{\kern-0.25em}
{\kern-0.2em}{\kern-0.2em}I}}

\def\wt{\widetilde}
\def\wh{\widehat}

\def\ov{\overline}

\usepackage{comment}

\long\def\forget#1\forgotten{} %

\newcommand\ver[1]{\marginpar{\tiny Changed in Ver \VER}}

\date{\today}

\usepackage{adjustbox}

\begin{document}

\begin{abstract}
Homotopic distance is a numerical homotopy invariant that quantifies the homotopic distinction between two or more continuous maps. In this paper, we look at various versions of homotopic distance between maps and study them on maps to group-like spaces and CW $H$-spaces. In particular, (1)~we develop the theories of probabilistic and diagonalized versions of the homotopic distance and compare their properties, and (2)~we show that all versions of the homotopic distance can be calculated precisely in terms of their respective versions of the Lusternik--Schnirelmann category when group-like spaces or CW $H$-spaces are involved. Our results are refined in the setting of rational groups. On the way, we also study loopings of maps and the nilpotency of the set of homotopy classes of maps to group-like spaces.
\end{abstract}

\title[Homotopic distances and group-like spaces]{Homotopic distances and group-like spaces}

\author[E.~Jauhari]{Ekansh~Jauhari}

\author[J.~Oprea]{John~Oprea}

\address{Ekansh Jauhari, Department of Mathematics, University of Florida, 358 Little Hall, Gainesville, FL 32611, USA.}

\email{ekanshjauhari@ufl.edu}

\address{John Oprea, Department of Mathematics, Cleveland State University, Cleveland, OH 44115, USA.}

\email{jfoprea@gmail.com}



\maketitle
\tableofcontents

\section{Introduction}\label{sec: intro}
How ``homotopically far apart'' are two given continuous maps $f,g\colon X \to Y$? The notion of \emph{homotopic distance}, first put forward in~\cite{MM}, gives a numerical answer to this question by recording the smallest size of open covers of the domain on each of whose members the two maps are homotopic. When $Y=X\times X$ and $f$ and $g$ are the factor inclusions, the homotopic distance between them is the \emph{Lusternik--Schnirelmann category} of $X$, a well-studied numerical homotopy invariant initially motivated from critical point theory and now crucial to the study of various topics in algebraic and geometric topology~\cite{CLOT}. On the other hand, when $X=Y\times Y$ and $f$ and $g$ are the factor projections, the homotopic distance between them is \emph{Farber's topological complexity} of $Y$, an invariant central to the study of the motion planning problem in topological robotics that is known to have connections with classical problems in geometry and topology~\cite{Farber book}. One gets similar interpretations via the notion of \emph{sequential homotopic distance} when considering a (finite) sequence of continuous maps. 
This unifying property of the homotopic distance is one of the several reasons to examine its behavior (especially on maps between spaces having interesting extra structures), and to study its other naturally motivated versions. In this paper, we pursue both of these directions of research. Our goal here is twofold: the first is to define and study new invariants related to sequential homotopic distance, and the second is to study the original and new invariants on maps to topological groups and $H$-spaces.

On (path-connected) topological groups and CW $H$-spaces, it is well-known that the topological complexity $\TC$, which is in general quite difficult to determine precisely, coincides with the LS-category $\cat$, which is more tame on such spaces, see~\cite{Far,LS}. This warrants the investigation of the general notion of the homotopic distance between maps when the codomain is such a \emph{group-like space}! In this setting, we explain that the sequential homotopic distance is completely understood in terms of the LS-category of a certain map. Furthermore, if we then enter the world of rational homotopy theory, we discover that the homotopic distance between (rationalized) maps to rational groups is measured simply by products of differences of the induced homomorphisms in rational cohomology, and therefore has a computationally effective algebraic description. Some auxiliary results that we obtain in this paper when path-connected topological groups $G$ are involved could be interesting in their own right. These include a new upper bound on the nilpotency class of the nilpotent group of homotopy classes of maps $X^m\to G$ that are null-homotopic on the diagonal $\Delta X\subset X^m$, and the existence of a new homotopy metric on the \emph{rationalized gauge groups} of principal $G$-bundles.

As for the invariants related to sequential homotopic distance, we define its probabilistic and diagonalized versions and obtain similar conclusions when group-like spaces are involved. Both LS-category and topological complexity have recently been re-imagined within a probabilistic framework, and this re-imagining has been dubbed as \emph{distributional category} and \emph{distributional topological complexity}, denoted by $\dcat$ and $\dTC$, respectively (see, for example,~\cite{DJ,KW}). The distributional invariants serve as lower bounds for the standard invariants 
and display interesting phenomena in their own right. For instance, while 
$\cat(\R P^n) = n$, we have $\dcat(\R P^n)=1$ for any $n$. But more is true: distributional invariants display a wholistic quality which differs from the 
reductionist approach to, for example, the LS-category. Indeed, the standard definition of $\cat(X)$ records one less than the smallest number of open sets 
that cover $X$ and whose inclusions are null-homotopic, and so in a sense, these open 
sets are the atoms that build the space, and the LS-category is a measure of the complexity
of assemblage. The distributional invariants, on the other hand, \emph{do not} have such an atomic formulation in terms of open sets, but rather count the least number of 
superpositions required to assemble ``information''. We will review the details in
Subsection \ref{subsec: secat}. 

In this direction, we show that the distributional version of the sequential homotopic distance interpolates between the distributional category and the sequential distributional complexity of spaces, and that it admits a general cohomological lower bound, both in analogy with the case of the original homotopic distance. For such results, we cannot use open sets, but must rely on more homotopical and classical methods. Because the distributional invariants force us to use a global approach, we also see the original homotopic distance and its characteristics from a new perspective. Along the way, we define the \emph{diagonal (distributional) homotopic distance} between maps and show that it similarly interpolates between the \emph{(distributional) one-category} and the \emph{diagonal (distributional) topological complexity} of spaces (see, for instance,~\cite{FGLO1,JO}). The new homotopic distances not only provide a unifying framework to study the respective distributional and diagonalized category and complexity-type invariants, but also shed new light on the original homotopic distance by displaying both similar and contrasting features.

\subsection*{Main results} 
Before stating our results, we declare that $\D(f_1,\dots,f_m)$ denotes the homotopic distance between $m$ maps $f_i\colon X\to Y$. We use the symbol $\D^{\cD}$ to talk about diagonal homotopy distances, and we write $\sd\D$ and $\sd\D^\cD$ to indicate that the same conclusions hold for \emph{both} the ordinary and the distributional invariants.

First, as mentioned above, we precisely determine the (distributional) homotopic distance $\sd\D$ in terms of the (distributional) category $\sd\cat$ of a map, and the diagonal (distributional) homotopic distance $\sd\D^\cD$ in terms of the (distributional) one-category $\sd\cat_1$ of the same map, when the codomain is a topological group.

\begin{theorem}[Theorems~\ref{thm: main1} and~\ref{thm: main2}]
    If $G$ is a path-connected topological group and $f_i\colon X\to G$ are maps 
for $1\le i \le m$, then
    \[
    \sd\D(f_1,\dots,f_m)=\sd\cat(F_m\circ (f_1,\ldots,f_m)\colon X\to G^{m-1}), \quad \text{and},
    \]
    \[
    \sd\D^{\cD}(f_1,\dots,f_m)=\sd\cat_1(F_m \circ (f_1,\ldots,f_m)\colon X\to G^{m-1}),
    \]
    where we have
$(f_1,\ldots,f_m)(x)= (f_1(x),\ldots,f_m(x)) \in G^m$ and $F_m(x_1,\dots,x_m)= (x_1x_2^{-1},x_2x_3^{-1},\dots,x_{m-1}x_m^{-1}) \in G^{m-1}$.
\end{theorem}
A statement like the above is more generally true for group-like spaces and (non-homotopy-associative) CW $H$-spaces (see Theorems~\ref{thm: main1} and~\ref{thm: main2}), and it recovers the main result of~\cite{LS}. Our proof involves the use of I.~M.~James's algebraic loop structure on the set of homotopy classes of maps from CW complexes to $H$-spaces, and the theory of homotopy pullbacks.

In the world of rational homotopy, we obtain a cohomological description.

\begin{theorem}[Theorem~\ref{thm:nilfgstar}]
    If $G$ is a compact topological group, $X$ is a simply connected finite CW complex, and $f_i\colon X\to G$ are maps for $1\le i \le m$, 
then
    \[
\D(f_{1\Q},\dots,f_{m\Q})=\cu_{\Q}\left(\Im\left(\Delta_{m-1}^*\circ 
\bigotimes_{i=1}^{m-1}\left(f_i^*-f_{i+1}^*\right)\right)
\right),
\]
where $f_{i\Q}\colon X_\Q \to G_\Q$ are the respective rationalized maps,  
$\Delta_{m-1}\colon X\to X^{m-1}$ is the diagonal map, and $\cu_\Q$ stands for the rational cup length. The same equality holds 
for $\dD(f_{1\Q},\dots,f_{m\Q})$ as well if $G$ is a compact Lie group.
\end{theorem}
Again, the above is true in a more general context, and its proof requires the use of tools from rational homotopy theory, such as the Sullivan 
minimal models and the \emph{rational category} of maps. As an application of this characterization, we also obtain a new invariant of gauge groups, see Section~\ref{sec:mapgroup}.

Next, it is a classical result due to G.~Whitehead that the homotopy classes of maps from an ANR space $X$ into a group-like space $G$ form 
a nilpotent group of nilpotency class bounded above by $\cat(X)$. While it is immediate that 
homotopic distance between maps is bounded above by the category of the domain $X$ of the maps, here we use the 
nilpotent structure of $[X,G]$ to refine this estimate. 

\begin{theorem}[Proposition~\ref{prop:catfbound}]
    If $X$ is an ANR, $G$ is a group-like space, and $f_i\colon X\to G$ are maps 
for $1\le i \le m$, then we have that
    \[
    \D(f_1,\dots,f_m)\le\cat(X)-\min\left\{\dep_{\cG}(f_i)\ \middle|\ 1\le i \le m\right\}    +1,
    \]
    where $\dep_{\cG}(f_i)$ is the depth\footnote{\hspace{0.5mm}The \emph{depth} of $f_i$ relative to a central series is defined in Subsection~\ref{subsec: whitehead}.} of $f_i$ in $\cG=[X,G]$ for each $i$, 
relative to the central series described in the proof of Theorem~\ref{thm:whitehead}.
\end{theorem}

As we show in Subsection~\ref{subsec: applications}, this estimate is sharp, and is obtained using the aforementioned interpretation of the homotopic distance in terms of the LS-category and the proof of Whitehead's theorem. Inspired by the latter, we also obtain a version of Whitehead's original theorem for the $m$-th sequential topological complexity $\TC_m$.

\begin{theorem}[Theorem~\ref{thm: tc whitehead}]
    If $G$ is a path-connected topological group and $X$ is an ANR, then the homotopy classes of maps $X^m\to G$ that are null-homotopic on the diagonal 
$\Delta X\subset X^m$ form a nilpotent group of nilpotency class at most $\TC_m(X)$.
\end{theorem}

The above estimate is a refinement since $\TC_m(X) \leq \cat(X^m)$ always holds, and strict equality is possible for several spaces, most notably for topological groups.

Our last main result is a general cohomological lower bound to the sequential homotopic distance that recovers a similar bound to the ordinary homotopic distance obtained in~\cite{MM}. Note that we also have a distributional version here.

\begin{theorem}[Remark~\ref{rem: dist lower bound} and Corollary~\ref{cor: same lower bound}]
If $R$ is a ring, $f_i\colon X\to Y$ are maps, and $c_i$ are integers for $1\le i \le m$ 
such that $\sum_{i=1}^mc_i=0$, then
\[
\D(f_1,\dots,f_m)\ge \cu_R\left(\Im\left(\sum_{i=1}^mc_if_i^*\right)\right),
\]
and the same lower bound holds for $\dD(f_1,\dots,f_m)$ in rational coefficients if $Y$ is a finite CW complex.
\end{theorem}
See Theorem~\ref{thm: lower bound} for a general cohomological lower bound to $\dD(f_1,\dots,f_m)$, whose proof involves the use of the (cohomology of) symmetric products of spaces. We also give several instances of the sharpness of these bounds in the paper and illustrate how they recover previously known cohomological lower bounds. 

In the rest of the paper, we study several properties and examples of (distributional) homotopic distance and diagonal (distributional) homotopic distance, their relationships with each other and with other \emph{sectional category}-type invariants, and their behavior under composition and looping of maps. As remarked earlier, our proofs use global homotopic methods rather than local ones.

\section{Basics}\label{sec:basics}
We use this section to revisit general facts about various sectional category-type invariants of our interest, and to make some topological observations.

\subsection{Concatenation of paths}\label{subsec: concat}
Given paths $f,g\in X^I$ for which $f(1)=g(0)$, the usual concatenation $f\cdot g\in X^I$ 
of $f$ and $g$ is the path
\[
f\cdot g(t)=\begin{cases}
    f(2t) & \text{ if }0\le t\le \tfrac{1}{2}
    \\
    g(2t-1) & \text{ if }\tfrac{1}{2}\le t\le 1.
\end{cases}
\]
In this paper, we will also use a more general (but also somewhat specific) concatenation 
of two or more paths, which we now describe.

For a space $X$, integer $m \ge 2$, and real numbers $a_i \in (1,\infty)$ such that $a_i > a_{i+1}$ for $1 \le i < m-1$, let
$$
T_{m}(X) = \left\{\left(f_{1}, \ldots, f_{m} \right) \in (X^I)^{m} \mid f_{i}(1) = f_{i+1}(0) \text{ for all } 1 \le i \le m-1 \right\},
$$
and let $\Theta_{m}: T_{m}(X) \to X^I$ be defined as $\Theta_{m}\left(f_{1}, \ldots, f_{m} \right) = f_{1} \star \cdots \star f_{m}$, where
$$
\left(f_{1} \star \cdots \star f_{m}\right)(t) = \begin{cases}
    f_{1}(a_{1}t) & \text{ if } 0 \le t \le \frac{1}{a_{1}} \\
    f_{2}\left(\frac{a_{2}(a_{1}t-1)}{a_{1}-a_{2}}\right) & \text{ if } \frac{1}{a_{1}} \le t \le \frac{1}{a_{2}}
    \\
    \hspace{6mm} \vdots & \hspace{8mm} \vdots 
    \\
    f_{m-1}\left(\frac{a_{m-1}(a_{m-2}\hspace{0.5mm} t-1)}{a_{m-2}-a_{m-1}}\right) & \text{ if } \frac{1}{a_{m-2}} \le t \le \frac{1}{a_{m-1}} 
    \\
    f_{m}\left(\frac{1-a_{m-1}\hspace{0.5mm} t}{1-a_{m-1}}\right) & \text{ if } \frac{1}{a_{m-1}} \le t \le 1.
\end{cases}
$$
It is easy to see using the definition of the compact-open topology on $X^I$ that the function $\Theta_{m}\colon T_m(X)\to X^m$ is continuous for each $m$, see~\cite[Lemma~2.1]{Ja1}. This concatenation operation will be quite useful for us.

\subsection{(Distributional) sectional category}\label{subsec: secat}
Suppose $p\colon E\to B$ is a (Hurewicz) fibration. Recall that the \emph{sectional category} of $p$, denoted $\secat(p)$, is the least non-negative integer $n$ such that $B$ can be covered by $n+1$ open subsets $U_j$ over each of which there exists a partial section of $p$, see~\cite{Sch}.

Let $\ast^k_BE$ denote the $k$-th iterated fiberwise join of $E$, i.e., 
\[
\ast^k_BE:=\left\{\lambda_1e_1+\cdots+\lambda_ke_k\ \middle| \ \sum_{i=1}^k\lambda_i=1,\lambda_i\ge 0, p(e_i)=p(e_j) \text{ for all }i,j \right\},
\]
and define the $k$-th fibration $\ast^k_Bp\colon \ast^k_BE\to B$ as $\ast^k_Bp(\lambda_1e_1+\cdots+\lambda_ke_k)=p(e_i)$. Of course, $\ast^k_BE$ is topologized with the subspace topology of the $k$-th iterated join $\ast^kE$, see also~\cite{CLOT}. Schwarz proved in~\cite{Sch} (see also~\cite{Jam}) that when $B$ is paracompact, $\secat(p)$ coincides with the smallest $n$ such that the $(n+1)$-th fibration $\ast^{n+1}_Bp$ admits a (continuous) section. This characterization, in part, motivates the following probabilistic constructions and definitions.

Given a fibration $p\colon E\to B$ and integer $k\ge 1$, let $\B_k(E)$ be the space of probability measures on $E$ having support at most $k$, and let $E_n(p)$ be the associated fiberwise subspace, i.e.,
\[
E_n(p):=\left\{\sum_{i=1}^k\lambda_i e_i \ \middle| \ \sum_{i=1}^k\lambda_i=1,\lambda_i\ge 0, p(e_i)=p(e_j)\text{ for all } i,j  \right\}.
\]
We note that the elements of $\ast^k_BE$ are \emph{ordered} formal linear combinations, while those of $E_n(p)$ are \emph{unordered} such linear combinations identified further by the relation $\lambda_1e+\lambda_2e\sim (\lambda_1+\lambda_2)e$ and hence written as probability measures with support size $|\supp(\mu)|\le k$. In particular, $E_n(p)$ has the subspace topology of the measure space $\B_k(E)$. The latter has a number of topologies to choose from; among them is the quotient topology coming from the topology of the symmetric join $\ast^kE/\Sigma_k$ (see, for instance,~\cite{KK}) and the metric topology coming from the L\'evy--Prokhorov (or the $1$-Wasserstein distance) metric --- see~\cite{DJ,KW,Dr} for related discussions and references. 

There is a natural map $p_k\colon E_k(p)\to B$ given by $p_k(\sum\lambda_ie_i)=p(e_i)$, which is, in fact, a fibration (see~\cite{DJ}). Then, in analogy with the equivalent definition of $\secat(p)$, one defines the \emph{distributional sectional category}\footnote{Strictly speaking, this invariant is called such when $E$ is metrizable and $\B_k(E)$ has the metric topology mentioned above. Otherwise, when $\B_k(E)$ has the quotient topology mentioned above, then such an invariant is called \emph{analog sectional category} in~\cite{KW,KW2} and denoted $\mathsf{asecat}(p)$. The inequality $\dsecat(p)\le\mathsf{asecat}(p)$ is known~\cite{Ja1,KW2}, and equality is conjectured. The results we prove in this paper are valid for the analog and the distributional invariants both; for simplicity and concreteness, we work only with the latter throughout the paper.} of $p$, denoted $\dsecat(p)$, as the least non-negative integer $n$ such that the $(n+1)$-th fibration $p_{n+1}$ admits a (continuous) section, see~\cite{DJ,KW,Ja1}. It can be shown from~\cite{KK} that 
\[
\dsecat(p)\le\secat(p)
\]
when $B$ is paracompact, see also~\cite{Dr,JO} and the references therein.

Moving forward, for simplicity, we will denote the fiberwise objects with $\ast^k_{\bullet}E$ and $\ast^k_{\bullet}p$ to avoid writing the base space when the context is clear, and their respective distributional (probabilistic) analogues with simply $E_n$ and $p_n$.

The fiberwise and the distributional constructions are both well-behaved under taking pullbacks and, more generally, with respect to maps between fibrations, see, for instance,~\cite{Sch,CLOT,DJ,KW,KW2,JO}. More precisely, if $p\colon E\to B$ and $p'\colon E'\to B'$ are fibrations with maps $f\colon E\to E'$ and $g\colon B'\to B$ such that $p'\circ f=g\circ p$, then for each $k\ge 1$, we have ``fiberwise joined maps'' 
\[
\ast^k_{\bullet}f\colon \ast^k_\bullet E\to \ast^k_\bullet E',\ \ \ast^k_{\bullet}p\colon \ast^k_\bullet E\to B,\  \text{ and } \ \ast^k_{\bullet}p'\colon \ast^k_\bullet E'\to B'
\]
such that $\ast^k_\bullet p'\circ \ast^k_\bullet f=g\circ \ast^k_\bullet p$. Similarly, there are the respective ``distributional maps'' $f_k\colon E_k\to E'_k$, $p_k\colon E_k\to B$, and $p_k'\colon E_k'\to B'$ such that $p_k'\circ f_k=g\circ p_k$.

\subsection{Categorical and complexity-type invariants}
Let us consider two specific fibrations. Let $(Y,y_0)$ be a pointed topological space, $P_0(Y)\subset Y^I$ be the space of paths with endpoint $y_0$, and $m\ge 2$ be an integer with an associated \emph{time schedule} 
\[
0=t_1<t_2<\cdots<t_{m-1}<t_m=1.
\]
Let $p^Y_m\colon P_0(Y)\to Y^{m-1}$ be the fibration that evaluates each $\phi\in Y^I$ at timestamps $t_i$ for $1\le i \le m-1$, and let $\pi_m^Y\colon Y^I\to Y^m$ be the fibration that evaluates each $\phi\in Y^I$ at the timestamps $t_i$ for $1\le i \le m$. We will follow this particular time schedule and the associated definition of $\pi^Y_m$ throughout the paper.

Suppose $f\colon (X,x_0)\to (Y,y_0)$ is a map, and for $r\ge 1$, let $f^r\colon X^r\to Y^r$ be the power map $f^r(x_1,\dots,x_r)=(f(x_1),\dots,f(x_r))$. The \emph{(distributional) Lusternik--Schnirelmann (LS) category} of $f^{m-1}$, denoted $\sd\cat(f^{m-1})$, is defined as 
\[
\sd\cat(f^{m-1}):=\sd\secat((f^{m-1})^*p^Y_m), 
\]
where $*$ in the superscript symbolizes a pullback, see~\cite{Jam,CLOT,DJ,KW,Dr,Ja1}. Similarly, the $m$\emph{-th sequential (distributional) topological complexity} of $f$, denoted $\sd\TC_m(f)$, is defined as 
\[
\sd\TC_m(f):=\sd\secat((f^m)^*\pi_m^Y),
\]
see~\cite{Far,Rud,BGRT,Sco,Kua,DJ,KW,Ja1,Ja2}. It follows that $\sd\cat(f^{m-1})\le\sd\TC_m(f)\le\sd\cat(f^m)$ for each $m\ge 2$.

If $Y$ is a connected, locally path-connected, and semi-locally simply connected space, then it admits a universal covering $\wt p^Y\colon \wt Y\to Y$. The \emph{(distributional) one-category} of $f$, denoted $\sd\cat_1(f)$, is defined as
\[
\sd\cat_1(f):=\sd\secat(f^*\wt p^Y),
\]
see~\cite{CLOT,Op,OS,JO}. Next, set $\pi=\pi_1(Y_,y_0)$, and let $\wh\pi^Y_m\colon \wh Y^m\to Y^m$ be a connected covering corresponding to the diagonal subgroup $\Delta\subset \pi^m$. Here, $\wh Y^m$ is the orbit space of the free diagonal action of $\pi$ on $\widetilde Y \times \cdots \times \widetilde Y=\Wi{Y}^m$, where the product is taken $m$-times. Then the $m$\emph{-th sequential diagonal (distributional) topological complexity} of $f$, denoted $\sd\TC^{\cD}_m(f)$, is defined as 
\[
\sd\TC^{\cD}_m(f):=\sd\secat((f^m)^*\wh\pi_m^Y),
\]
see~\cite{FGLO1,FGLO2,PS,JO}. Again, we have for each $m\ge 2$ that $\sd\cat_1(f^{m-1})\le\sd\TC^{\cD}_m(f)\le\sd\cat_1(f^m)$. Also, we have $\sd\cat_1(f)\le\sd\cat(f)$ and $\sd\TC_m^{\cD}(f)\le \sd\TC_m(f)$; see Section~\ref{sec:dhd} for details.

\section{Sequential (distributional) homotopic distance}\label{sec: homotopic distance}
Let $m\ge 2$ be an integer and $f_i\colon X\to Y$ be continuous maps. Then the $m$\emph{-th sequential homotopic distance} between $\{f_i\}_{i=1}^m$, denoted $\D(f_1,\ldots,f_m)$, is defined as the least non-negative integer $n$ such that $X$ can be covered by $n+1$ open sets $U_j$ such that for each $j$, $f_{i|U_j}\simeq f_{k|U_j}$ for $1\le i,k\le m$, see~\cite{MM}.

It is clear that $\D(f_1,\ldots,f_m)=0$ if and only if $f_i\simeq f_k$ for all $i,k$. Essentially, $\D(f_1,\ldots,f_m)$ gives a numerical measure of how close the maps $f_1,\ldots,f_m$ are to being mutually homotopic. Note also that $\D(f_1,\ldots,f_m)=\D(f_{\sigma(1)},\dots,f_{\sigma(m)})$ for any permutation $\sigma\in S_m$ on $m$ symbols.

Let $(f_1,\dots, f_m)\colon X\to Y^m$ be the map $(f_1,\dots,f_m)(x)=(f_1(x),\dots,f_m(x))$. We will follow this notation throughout the paper. 

The following was observed in~\cite{MM,BV}; we record a simple proof.

\begin{lemma}\label{lem: dist as secat}
    Given $m\ge 2$ and maps $f_i\colon X\to Y$ for $1\le i \le m$, we have that $\D(f_1,\dots,f_m)=\secat(q^X)$, where $q^X$ is the pullback $(f_1,\dots,f_m)^*\pi^Y_m$ as follows:
\begin{equation}\label{diag: main pullback}
\xymatrix@C=2cm{
Q \ar[r]^-{} \ar[d]_-{q^X} & Y^I \ar[d]^-{\pi^Y_m} \\
X \ar[r]^-{(f_1,\dots,f_m)} & Y^m.
}   
\end{equation}
\end{lemma}

\begin{proof}
We have $Q=\{(x,\alpha)\in X\times Y^I\mid \alpha(t_i)=f_i(x)\ \text{for all}\ i\}$. Fix a subset $U\subset X$. Suppose $f_{i|U}\simeq f_{k|U}$ for all $i$ and $k$. Then there exists a homotopy $H_i\colon U\to Y^I$ from $f_i$ to $f_{i+1}$ for $1\le i \le m-1$. Define a map $H\colon U\to Y^I$ such that for each $x\in U$, the image $H(x)$ is the path concatenation
\[
H(x)=H_1(x)\star H_2(x)\star\cdots\star H_{m-1}(x)
\]
from Subsection~\ref{subsec: concat}, where we set $a_i=\tfrac{1}{t_{i+1}}$ for $1\le i\le m-1$. Then this path satisfies $H(x)(t_i)=f_i(x)$ for $1\le i \le m$, so that the map $\tau\colon U\to Q$ given by $\tau(x)=(x,H(x))$ is a partial section of $q^X$. Conversely, suppose $\rho\colon U\to Q$ is a partial section of $q^X$ and $\rho(x)=(x,\alpha_x)$ for each $x$. For $1\le i\le m-1$, define $H_i\colon U \times I\to Y$ as
\[
H_i(x,t)=\alpha_x(t(t_{i+1}-t_i)+t_i).
\]
This shows that $f_{i|U}\simeq f_{i+1|U}$ for each $i$.
\end{proof}

Inspired by Lemma~\ref{lem: dist as secat}, we define the $m$\emph{-th sequential distributional homotopic distance} between maps $\{f_i\}_{i=1}^m$, denoted $\dD(f_1,\dots,f_m)$, as
\[
\dD(f_1,\dots,f_m):=\dsecat((f_1,\dots,f_m)^*\pi^Y_m).
\]

\subsection{Standard properties}
Let us now explain how sequential (distributional) homotopic distance between maps recovers previously studied invariants (distributional) LS-category and sequential (distributional) topological complexity of a map. In the original setting, such results for the identity map appear in~\cite{MM,BV}.

\begin{proposition}\label{prop: recovering tcf}
    Let $X$ be a space and let $\pr_i\colon X^m\to X$ be the projection onto the $i$-th coordinate. If $f\colon X\to Y$ is a map, then 
    \[
    \sd\cat(f)=\sd\D(f,\ast)=\sd\D(\ast,f),
    \]
    where $\ast\colon X\to Y$ is the constant map to a fixed basepoint $y_0\in Y$, and
    \[
    \sd\TC_m(f)=\sd\D(f\circ \pr_1,\dots,f\circ\pr_m).
    \]
\end{proposition}

\begin{proof}
    First, note that the pullback space corresponding to $(f,\ast)^*\pi^Y_2$ is
    \begin{align*}
        A & = \left\{ (x,\alpha)\in X\times Y^I\ \middle|\ \alpha(0)=f(x),\alpha(1)=y_0 \right\}
        \\
        &=\left\{ (x,\alpha)\in X\times P_0(Y)\ \middle|\ \alpha(0)=f(x) \right\}.
    \end{align*}
    Then $A$ is also the pullback space corresponding to $f^*p^Y_2$. In particular, the pullbacks $(f,\ast)^*\pi^Y_2$ and $f^*p^Y_2$ coincide. Therefore, by definition, we have that $\sd\D(f,\ast)=\sd\secat((f,\ast)^*\pi^Y_2)=\sd\secat(f^*p^Y_2)=\sd\cat(f)$. An entirely analogous argument yields the other equality $\sd\D(\ast,f)=\sd\cat(f)$.
    
    Next, since $(f\circ \pr_1,\dots,f\circ\pr_m)=f^m$, the pullbacks $(f\circ \pr_1,\dots,f\circ\pr_m)^*\pi^Y_m$ and $(f^m)^*\pi^Y_m$ coincide.
    Therefore, we get $\sd\TC_m(f)=\sd\secat((f^m)^*\pi^Y_m)=\sd\secat((f\circ \pr_1,\dots,f\circ\pr_m)^*\pi^Y_m)=\sd\D(f\circ \pr_1,\dots,f\circ\pr_m)$ by definition.
\end{proof}

Because $\sd\secat(f^*p)\le\sd\secat(p)$ for any fibration $p\colon E\to B$ and map $f\colon A\to B$ (see~\cite{Sch,JO}), we have $\sd\D(f_1,\dots,f_m)\le\sd\TC_m(Y)$ for maps $f_i\colon X\to Y$ by definition. This upper bound can be improved as follows.

\begin{proposition}\label{prop: (co)domain}
    For $m\ge 2$ and based maps $f_i\colon (X,x_0)\to (Y,y_0)$ for $1\le i \le m$, we have that
    \[
    \sd\D(f_1,\dots,f_m)\le\min\left\{\sd\cat(X),\ \prod_{i=1}^m\left(\sd\cat(f_i)+1\right)-1\right\}.
    \]
\end{proposition}

The proof of this proposition is constructive, so we show it below only for the (new) distributional invariants using measures of paths, noting that the same proof using instead ordered linear combinations of paths (in the setting of the iterated fiberwise join) will yield the corresponding inequalities for the original invariants.
\begin{proof}
    Let $\dcat(X)=n$, so that there exists a map $s\colon X\to (P_0(X))_{n+1}$. Given some $x\in X$ and $\phi\in \supp(s(x))$ with corresponding weight $\lambda_\phi$, define a path $\gamma_{i}:=f_i\circ\phi\in Y^I$ for $1\le i \le m$, and the concatenated path
    \[
\psi_{\phi}:=\bigl(\gamma_{1}\cdot\ov{\gamma_{2}}\bigr)\star\bigl(\gamma_{2}\cdot\ov{\gamma_{3}}\bigr)\star\cdots\star\bigl(\gamma_{{m-1}}\cdot\ov{\gamma_{m}}\bigr),
    \]
    where $\cdot$ denotes the usual concatenation of paths, and the concatenation using $\star$ is from Subsection~\ref{subsec: concat}, where, as before, we set $a_i=\tfrac{1}{t_{i+1}}$ for $1\le i \le m-1$. Then the map $H\colon X\to Q_{n+1}$ given by 
    \[
H(x)=\left(x,\sum_{\phi\in\supp(s(x))}\lambda_{\phi}\psi_{\phi}\right)
    \]
    is a section of the fibration $q^X_{n+1}$ induced by the fibration $q^X=(f_1,\dots,f_m)^*\pi^Y_m$ described in~\eqref{diag: main pullback}. This proves that $\dD(f_1,\dots,f_m)\le n=\dcat(X)$.

    Next, suppose $\dcat(f_i)=n_i$ for each $1\le i \le m$, so that there exist $m$ maps $s_i\colon X\to (P_0(Y))_{n_i+1}$. Given $x\in X$ and $\phi_i\in \supp(s_i(x))$ with corresponding weight $\lambda_{\phi_i}$ for each $i$, define $\phi_{1,\dots,m}\in Y^I$ as the concatenation
    \[
\phi_{1,\dots,m}:=\bigl(\phi_{1}\cdot\ov{\phi_{2}}\bigr)\star\bigl(\phi_{2}\cdot\ov{\phi_{3}}\bigr)\star\cdots\star\bigl(\phi_{{m-1}}\cdot\ov{\phi_{m}}\bigr).
    \]
Because $|\supp(s_i)(x)| \leq n_i+1$, there are $n=\prod_{i=1}^m(n_i+1)$ ways to choose
a path $\phi_{1,\ldots,m}$, and we can assemble them all in a map $H\colon X\to Q_{n}$ given by 
    \[
H(x)=\left(x,\sum_{(\phi_1,\dots,\phi_m)\in(\supp(s_1(x)),\dots,\supp(s_m(x)))}\left(\prod_{i=1}^m\lambda_{\phi_i}\right)\phi_{1,\dots,m}\right).
    \]
    This a section of the fibration $q^X_{n}$, so that $\dD(f_1,\dots,f_m)\le n-1$, as desired.
\end{proof}

In particular, $\dTC_m(X)\le \prod_{i=1}^m(\dcat(\pr_i)+1)-1\le(\dcat(X)+1)^m-1$ for the projections $\pr_i\colon X^m\to X$, so we recover the general upper bound from~\cite{Ja1}.

\begin{proposition}\label{prop: recovering cat}
    Let $f\colon(X,x_0)\to (Y,y_0)$ be a map. If $\iota_i\colon X^{m-1}\to X^m$ is the map $\iota_i(x_1,\dots,x_{m-1})=(x_1,\dots,x_{i-1},x_0,x_{i},\dots,x_{m-1})$ for $1\le i \le m$, then
    \[
    \sd\cat(f^{m-1})\le\sd\D(f^m\circ\iota_1,\dots,f^m\circ\iota_m).
    \]
    In particular, $\sd\cat(X^{m-1})=\sd\D(\iota_1,\dots,\iota_m)$.
\end{proposition}

\begin{proof}
In the setting of~\eqref{diag: main pullback}, the pullback for $(f^{m}\circ\iota_1,\dots,f^{m}\circ\iota_m)^*\pi^{Y^m}_m$ is
\[
\xymatrix@C=3cm{
Q \ar[r]^-{} \ar[d]_-{q^{X^{m-1}}} & (Y^m)^I \ar[d]^-{\pi^{Y^m}_m} \\
X^{m-1} \ar[r]^-{(f^m\circ\iota_1,\dots,f^m\circ\iota_m)} & (Y^m)^m,
}   
\]
and the pullback space $Q$ is
\begin{multline*}
Q=\{(x_1,\dots,x_{m-1},\alpha_1,\dots,\alpha_{m})\in X^{m-1}\times (Y^{I})^{m} \ \mid \ \alpha_i(t_i)=y_0, 
\\
\ \alpha_i(t_j)=f(x_j) \ \text{if}\ j<i, \ \alpha_i(t_j)= f(x_{j-1})\text{ if }j > i\text{ for } 1\le i\le m\},
\end{multline*}  
where we use $(Y^m)^I\cong (Y^I)^m$. Consider the map $\psi\colon Q\to P_0(Y)$ defined as
\[
\psi(x_1,\dots,x_{m-1},\alpha_1,\dots,\alpha_{m})=\alpha_m.
\]
Then we have the following commutative diagrams for each $k\ge 1$, induced by the fiberwise measure (on the left) and the fiberwise join (on the right) constructions from the diagram $f^{m-1}\circ q^{X^{m-1}}=p_m^Y\circ \psi$:
\[
\xymatrix@C=1.5cm{
Q_k \ar[r]^-{\psi_k} \ar[d]_-{q^{X^{m-1}}_k} 
& 
P_0(Y)_k \ar[d]^-{(p_m^{Y})_k}
\\
X^{m-1}\ar[r]^-{f^{m-1}} 
&
Y^{m-1}, 
}
\quad \quad \xymatrix@C=1.5cm{
\ast^k_{\bullet} Q \ar[r]^-{\ast^k_{\bullet}\psi} \ar[d]_-{\ast^k_{\bullet} q^{X^{m-1}}} 
& 
\ast^k_{\bullet} P_0(Y) \ar[d]^-{\ast^k_{\bullet} p_m^{Y}}
\\
X^{m-1}\ar[r]^-{f^{m-1}} 
&
Y^{m-1}.
}
\]
In either diagram, a given section of the left map, when composed with the top map, produces a lift of $f^{m-1}$ along the right map. So,
\[
\sd\cat(f^{m-1}) = \sd\secat((f^{m-1})^*p^Y_m) \le \sd\secat(q^{X^{m-1}}) = \sd\D(f^m\circ\iota_1,\dots,f^m\circ\iota_m),
\]
where the last equality is due to Lemma~\ref{lem: dist as secat} and the definition of $\dD$. Upon taking $f=\id_X$ and using the first upper bound from Proposition~\ref{prop: (co)domain}, we get $\sd\cat(X^{m-1})=\sd\D(\iota_1,\dots,\iota_m)$.
\end{proof}

The equality $\cat(X^{m-1})=\D(\iota_1,\dots,\iota_m)$ was first noted in~\cite{MM,BV}.

\begin{ex}
Let $f_i\colon X\to Y$ be any maps for $1\le i \le m$ such that one of $X$ or $Y$ is 
an $(\R P^n)^{m-1}$.
Note that $\dD(f_1,\dots,f_m)\le\dcat((\R P^n)^{m-1})$ by Proposition~\ref{prop: (co)domain}. Also, 
\[
\dcat((\R P^n)^{m-1})\le(\dcat(\R P^n)+1)^{m-1}-1=2^{m-1}-1
\]
in view of~\cite[Proposition~3.5]{DJ}. So, $2^{m-1}-1$ is a \emph{universal} upper bound to $\dD(f_1,\dots,f_m)$. However, $\D(\iota_1,\dots,\iota_m)=\cat((\R P^n)^{m-1})=n(m-1)$ by Proposition~\ref{prop: recovering cat} for the maps $\iota\colon (\R P^n)^{m-1}\hookrightarrow (\R P^n)^m$. So, we can take large values of $n$ and see arbitrarily large gaps between the distributional and the usual homotopic distances. 
This can also be seen by taking, more generally, one of $X$ or $Y$ to be a sufficiently 
high-dimensional lens space or a projective product space in view of the results from~\cite[Section~8]{JO}.
\end{ex}

\subsection{Covering spaces}
As explained above, (distributional) homotopic distance ``interpolates'' between (distributional) LS-category and (distributional) topological complexity. Now, for a covering space $\alpha\colon \overline X \to X$, it is easy to show that $\sd\cat(\overline X) \leq \sd\cat(X)$~\cite{CLOT,DJ}, but $\TC(\overline X) \leq \TC(X)$ is \emph{not}
true in general, see~\cite{Dr-covering}. This means that there can
be no analogue of the $\sd\cat$ inequality for $\D$ in general. Indeed, if we lift the projections $\pr_i\colon X \times X \to X$
to the covering space $\overline X \times \overline X \to \overline X$, we 
simply obtain the respective projections $\wt \pr_1, \wt \pr_2$, so that the homotopic
distance of the lifts is simply $\sd\TC(\overline X)$. Here we shall show that the 
same example, due to Dranishnikov~\cite{Dr-covering}, with $\TC(\overline X) > \TC(X)$ also gives
the analogous result for distributional complexity. Hence, there
can be no general result for distributional homotopic distance as well.

Recall the following result from~\cite[Subsection 6.3]{DJ} (see also~\cite{Dr}).
\begin{theorem}\label{thm:wedge}
If $X$ and $Y$ are aspherical complexes that satisfy
$$\max \{\dim X, \dim Y\} < \max \{\dTC(X), \dTC(Y), \dcat(X \times Y)\},$$
then 
$$\dTC(X \vee Y) = \max \{\dTC(X), \dTC(Y), \dcat(X \times Y)\}.$$
\end{theorem}

\begin{ex}\label{exam:nocover}
Let us take $X=S^1$ and $Y=T^2$ in the above theorem. Note that $Z:= X\vee Y=S^1 \vee T^2$ satisfies the hypotheses of the above theorem
since the maximum dimension is $2$, while $\dcat(S^1 \times T^2)=\dcat(T^3)=3$ because of the fact that rational cup length of a finite CW complex bounds its $\dcat$ value from below, see~\cite[Subsection~4.2]{DJ}. Therefore, $\dTC(Z)=3$. Now, take the covering space given by the $2$-fold cover of $S^1$ by itself with 
two tori attached: $\overline Z = T^2 \vee S^1 \vee T^2$. To apply the theorem,
we note that $T^2 \vee S^1$ and $T^2$ are both aspherical of dimension $2$, while
$$\dcat((T^2 \vee S^1) \times T^2) \geq \cu_\Q((T^2 \vee S^1)\times T^2)
= 4.$$
But then we have 
\[
\dTC(\overline Z) = \max \{3,2,\dcat((T^2 \vee S^1) \times T^2)\}
\geq 4
\]
by Theorem~\ref{thm:wedge}. Thus, in general, we do not have $\dTC(\overline Z) \leq \dTC(Z)$ for an arbitrary covering $\overline Z \to Z$.
\end{ex}

We ask whether there exist any conditions on maps $f_i\colon X \to Y$ such that
$\sd\D(\wt f_1, \dots,\wt f_m) \leq \sd\D(f_1,\dots,f_m)$, where $\wt f_i\colon \overline X \to \overline Y$ is the lift of $f_i$ for each $i$.

\section{Sequential diagonal (distributional) homotopic distance}\label{sec:dhd}

In this section, we introduce a new diagonalized invariant of homotopic distance and its distributional analogue, and study their properties.

\begin{definition}\label{def:tcd} 
Let $m\ge 2$ be an integer and $Y$ be a path-connected space with fundamental group
$\pi = \pi_1(Y,y_0)$. The $m$\emph{-th sequential diagonal (or $\cD$-)homotopic distance} between maps $f_i\colon X \to Y$ for $1\le i \le m$, denoted $\D^\cD(f_1,\dots,f_m)$, is defined as the least non-negative integer $n$ such that $X$ can be covered by $n +1$ 
open subsets $U_j$ such that for each $j$ and any choice of basepoint $u_j\in U_j$, the map 
\[
\pi_1\left(U_j, u_j\right) \longrightarrow 
\pi_1\left(Y^m, \left(f_1(u_j),\dots,f_m(u_j)\right)\right),
\]
induced by the restriction $(f_1,\dots,f_m)_{|U_j}\colon U_j\to Y^m$, takes
values in a subgroup conjugate to the diagonal subgroup $\Delta < \pi^m$.
\end{definition}

Here, we mention that for each $u_j\in U_j$, there is an isomorphism
\[
\pi_1(Y^m, (f_1(u_j),\dots,f_m(u_j))) \to \pi_1(Y^m, (y_0, \dots,y_0)) = \pi^m
\]
determined uniquely up to conjugation, and that the diagonal inclusion $Y \to Y^m$ induces
the inclusion $\pi \to \pi^m$ onto the diagonal subgroup $\Delta$.

\begin{remark}\label{rem:zerocrit}
There are two things to note about Definition \ref{def:tcd}. First, the 
criterion that $(f_1,\dots, f_m)_\#$ maps $\pi_1(U_j)$ to a subgroup conjugate to
the diagonal subgroup is equivalent to saying that the restrictions $(f_i)_{\#|\pi_1(U_j)}$ and $(f_k)_{\#|\pi_1(U_j)}$ are conjugate for $1\le i,k\le m$. Second, this reformulation says that $\D^\cD(f_1,\dots,f_m)=0$ if and only if, up to conjugation, $(f_i)_\#=(f_k)_\#$ for $1\le i,k\le m$ on all of $\pi_1(X)$.
\end{remark}

For connected, locally path-connected, and semi-locally simply connected $Y$, recall that $\wh\pi^Y_m\colon \wh Y^m\to Y^m$ is a connected covering corresponding to the diagonal subgroup $\Delta< \pi^m$. In analogy with Lemma~\ref{lem: dist as secat}, we have the following.

\begin{lemma}\label{lem: cD-dist as secat}
    Given $m\ge 2$ and maps $f_i\colon X\to Y$ for $1\le i \le m$, we have that $\D^{\cD}(f_1,\dots,f_m)=\secat(\wh q^X)$, where $\wh q^X$ is the pullback $(f_1,\dots,f_m)^*\wh\pi^Y_m$ as follows:
\begin{equation}\label{diag: second main pullback}
\xymatrix@C=2cm{
\wh Q \ar[r]^-{} \ar[d]_-{\wh q^X} & \wh Y^m \ar[d]^-{\wh\pi^Y_m} \\
X \ar[r]^-{(f_1,\dots,f_m)} & Y^m.
}   
\end{equation}
\end{lemma}

\begin{proof}
The result follows by noticing two things. First, if $U\subset X$ is open, 
then $\pi_1(U) \to \pi_1(Y^m)$ is conjugate to a subgroup
of the diagonal $\Delta$ if and only if the composition 
\[
U \hookrightarrow X \xrightarrow{(f_1,\dots,f_m)} Y^m
\]
lifts to a map $U \to \wh Y^m$ along $\wh\pi^Y_m$. In turn, because of the pullback property, such a lift exists 
if and only if there is a partial section $s\colon U \to \wh Q$ to $\wh q^X$.
\end{proof}

\begin{remark}\label{rem:surject}
We note that the pullback space $\wh Q$ in~\eqref{diag: second main pullback} will \emph{not} be connected, unless $\Im((f_1,\dots,f_m)_\#)$ and $\Delta$ 
generate all of $\pi_1(Y^m)\cong \pi_1(Y)^m$. This follows from the
Mayer--Vietoris sequence for homotopy groups of a pullback. This does not seem to 
present any problems in formulating our results. Indeed, for ordinary LS-category and 
sectional category, only a few results (such as the bound on the category of the 
total space of a fibration) depend on the connectedness of the space.
\end{remark}

With Lemma~\ref{lem: cD-dist as secat} in view, we define the following.

\begin{definition}\label{def: d-diagonal dist}
For an integer $m\ge 2$, the $m$\emph{-th sequential distributional diagonal (or $\cD$-)homotopic distance} between maps $\{f_i\}_{i=1}^m$, denoted $\dD^{\cD}(f_1,\dots,f_m)$, is defined as
\[
\dD^{\cD}(f_1,\dots,f_m):=\dsecat((f_1,\dots,f_m)^*\wh\pi^Y_m).
\]    
\end{definition}

Diagonal (distributional) homotopic distance between maps recovers the previously studied invariants sequential diagonal (distributional) topological complexity and (distributional) one-category of a map as follows.

\begin{proposition}\label{prop: recovering cd-tcf}
    Let $X$ be a space and let $\pr_i\colon X^m\to X$ be the projection onto the $i$-th coordinate. If $f\colon X\to Y$ is a map, then 
    \[
    \sd\TC^{\cD}_m(f)=\sd\D^{\cD}(f\circ \pr_1,\dots,f\circ\pr_m).
    \]
    Moreover, if $h\colon Y\to Z$ is a map that induces an isomorphism of fundamental groups and $\ast\colon Y\to Z$ is the constant map to a basepoint $z_0\in Z$, then
    \[
    \sd\cat_1(f)=\sd\D^{\cD}(h\circ f,\ast\circ f)=\sd\D^{\cD}(\ast\circ f,h\circ f).
    \]        
\end{proposition}

\begin{proof}
The proof for $\sd\TC^{\cD}_m$ is entirely analogous to that for $\sd\TC_m$ from Proposition~\ref{prop: recovering tcf}. For $\sd\cat_1$, we consider the following two pullbacks, where the left one is for $\sd\cat_1(f)$ and the right one is in the setting of~\eqref{diag: second main pullback}:
\[
\xymatrix@C=1.2cm{
P \ar[r]^-{} \ar[d]_-{f^*\wt p^Y} & \wt Y \ar[d]^-{\wt p^Y} \\
X \ar[r]^-{f} & Y,
} \quad \quad \quad \xymatrix@C=1.2cm{
\wh Q \ar[r]^-{} \ar[d]_-{\wh q^Y} & \wh Z^2 \ar[d]^-{\wh\pi^Y_2} \\
Y \ar[r]^-{(h,\ast)} & Z^2.
} 
\]
Because $h_\#$ is an isomorphism, $\wh q^Y\colon\wh Q\to Y$ is a covering whose fiber is bijective to $\pi_1(Y)$. Since $\Im(h_\#)$ and $\Delta_Z$ clearly generate $\pi_1(Z)^2$, the cover $\wh Q$ is connected (\emph{cf}. Remark~\ref{rem:surject}), and is therefore the universal cover $\wt Y$ of $Y$. Thus, the maps $\wh q^Y$ and $\wt p^Y$ coincide, so the above two (homotopy) pullbacks can be merged along the middle vertical map to yield the following (homotopy) pullback:
\[
\xymatrix@C=2cm{
P \ar[r]^-{} \ar[d]_-{f^*\wt p^Y} & \wh Z^2 \ar[d]^-{\wh\pi^Y_2} \\
X \ar[r]^-{(h\circ f,\ast\circ f)} & Z^2.
} 
\]
It then follows that $f^*\wt p^Y\simeq (h\circ f,\ast\circ f)^*\wh\pi_2^Y$. Finally, the homotopy invariance of $\sd\secat$ gives us the desired equality
\[
\sd\cat_1(f)=\sd\secat(f^*\wt p^Y)=\sd\secat((h\circ f,\ast\circ f)^*\wh\pi_2^Y)=\sd\D^{\cD}(h\circ f,\ast\circ f).
\]
Indeed, the first equality is by the definition of $\sd\cat_1$ and the last one is by the definition of $\sd\D^{\cD}$, see Lemma~\ref{lem: cD-dist as secat} and Definition~\ref{def: d-diagonal dist}. The proof of the other equality $\sd\cat_1(f)=\sd\D^{\cD}(\ast\circ f,h\circ f)$ is similar.
\end{proof}

\begin{ex}\label{exam:TtimesCP}
Let $X=Y=Z=T^{2k}\times \C P^{n-k}$ with $n\ge k$ and $f=h=\id_Y$. Then Propositions~\ref{prop: recovering tcf} and~\ref{prop: recovering cd-tcf} give 
\[
\D^\cD(\id_Y,*)=\cat_1(Y)=2k\quad \text{and}\quad \sd\D(\id_Y,*)=\sd\cat(Y)=n+k,
\]
see~\cite{Op,DJ}. If $k=1$, then $\dD^\cD(\id_Y,*)=\dcat_1(Y)=2$ by~\cite[Section~4]{JO}. Thus, there are several cases when $\sd\D^{\cD}$ and $\sd\D$ are arbitrarily far apart.
\end{ex}

Also, as before, we get $\sd\D^{\cD}(f_1,\dots,f_m)\le\sd\TC^{\cD}_m(Y)$ for maps $f_i\colon X\to Y$ by definition. Next, we notice the following.

\begin{proposition}\label{prop: cd-D vs D}
Given maps $f_i\colon X\to Y$ for $1\le i \le m$, we have that
\[
\sd\D^{\cD}(f_1,\dots,f_m)\le \sd\D(f_1,\dots,f_m).
\]
This inequality is saturated if $Y$ is an Eilenberg--Mac~Lane space $K(\pi,1)$.
\end{proposition}
\begin{proof}
First, consider the map $\theta_m\colon Y^I\to \wh Y^m$ given by
\[
\theta_m(\gamma) = \left[\gamma(0),\gamma(t_2),\ldots,\gamma(t_{m-1}),\gamma(1)\right],
\]
which is well-defined because we simply quotient by the $\pi$-action. As explained in~\cite{FGLO2}, $\wh \pi^Y_m\circ \theta_m=\pi^Y_m$. Then we get the following commutative cube,
\begin{equation}\label{diag: cube}
\xymatrix@C=3pc{
Q \ar[dr]_-{q^X} \ar[dd]_-{\psi} \ar[rr] & & Y^I \ar[dd]^(.6){\theta_m} \ar[dr]^-{\pi^Y_m} &\\
& X \ar[rr]^(.3){(f_1,\dots,f_m)} \ar[dd]_(.3){\id_X} & & Y^m \ar[dd]^-{\id_{Y^m}} \\
\wh Q \ar[rr] \ar[dr]_-{\wh q^X}\ & & \wh Y^m \ar[dr]^-{\wh\pi^Y_m} \\
& X \ar[rr]^-{(f_1,\dots,f_m)} && Y^m,
}
\end{equation}
where the squares on the top and bottom are pullbacks, and $\psi\colon Q\to\wh Q$ is the whisker map given by the universal property of the pullback at the bottom. The commutative square on the left then gives for each $k\ge 1$ a commutative triangle
\[
\xymatrix{
Q_k \ar[dr]_-{q^X_k} \ar[rr]^-{\psi_k} & & \wh Q_k \ar[dl]^-{\wh q^X_k} \\ & X. &
}
\]
Clearly, a section of $q^X_k$ produces that of $\wh q^X_k$. Similarly, one concludes that a section of $\ast^k_\bullet q^X$ produces that of $\ast^k_\bullet \wh q^X$ using the fiberwise joined map $\ast^k_\bullet\psi$. Therefore,
\[
\sd\D^{\cD}(f_1,\dots,f_m) = \sd\secat(\wh q^X) \le \sd\secat(q^X) = \sd\D(f_1,\dots,f_m).
\]
Now, if $Y$ is an aspherical locally finite CW complex, then $\theta_m$ is a fiber homotopy equivalence, see~\cite{FGLO2}. Hence, $\pi^Y_m$ can be replaced by the homotopy equivalent map $\wh\pi^Y_m$, so that up to homotopy, the pullback squares at the top and bottom of~\eqref{diag: cube} coincide. This gives the desired equality $\sd\secat(\wh q^X) = \sd\secat(q^X)$.
\end{proof}

Before we present the next result, we recall a version of a result from \cite{GLO} rephrased in our language. Recall that two subgroups $A,B < G$ are \emph{complementary} if $G=AB$ and $A \cap B=\{e\}$.

\begin{lemma}\label{lem:catX1X2}
Given maps $f_i\colon X_i \to Y$ of connected spaces for $i\in\{1,2\}$, suppose
\begin{enumerate}
\item each $(f_i)_\# \colon \pi_*(X_i) \to \pi_*(Y)$ is an inclusion on 
all homotopy groups;
\item there is a direct sum decomposition 
$(f_1)_\# (\pi_j(X_1)) \oplus (f_2)_\# (\pi_j(X_2)) \cong \pi_j(Y)$ for each $j \geq 2$;
\item  $(f_1)_\# (\pi_1(X_1))$ and $(f_2)_\# (\pi_1(X_2))$ 
are complementary subgroups in  $\pi_1(Y)$.
\end{enumerate}
Then the pullback space $P$ corresponding to the evaluation fibration $\pi^Y_m\colon Y^I \to Y^m$ over $f_1\times f_2\times\id_{Y^{m-2}}\colon X_1\times X_2 \times Y^{m-2}\to Y^m$
is contractible.
\end{lemma}

To put this in our context, given maps $f_i\colon X_i \to Y$ for $i\in\{1,2\}$, define auxiliary maps 
$\tilde f_j\colon X_1\times X_2 \times Y^{m-2} \to Y$ for $1\le j \le m$ as $\tilde f_j(x_1,x_2,y_1,\dots,y_{m-2})= f_j(x_j)$ for $j\in\{1,2\}$ and $\wt f_j(x_1,x_2,y_1,\dots,y_{m-2})=y_{j-2}$ for $3\le j \le m$. We then obtain the following.

\begin{theorem}\label{thm:mapdist}
If the maps $f_i \colon X_i \to Y$ for $i\in\{1,2\}$ satisfy the conditions of Lemma \ref{lem:catX1X2}, 
then the homotopic distance between the maps $\{\wt f_j\}_{j=1}^m$ is given by
\[
\sd\D(\wt f_1,\dots,\wt f_m) = \sd\cat(X_1 \times X_2\times Y^{m-2}).
\]
Moreover, if $Y=K(\pi,1)$ is aspherical, then $\sd\D(\wt f_1,\dots,\wt f_m) = \sd\D^\cD(\wt f_1,\dots,\wt f_m)$. In particular, if $X_i=K(A_i,1)$ for $A_i=(f_i)_\# \bigl(\pi_1(X_i)\bigr)$ for $i\in\{1,2\}$, then
\[
\D(\wt f_1,\dots,\wt f_m) = \D^\cD(\wt f_1,\dots,\wt f_m) = \cd(A_1 \times A_2\times \pi^{m-2}),
\]
and if, additionally, $A_i$ and $\pi$ are torsion-free, then 
\[
\dD(\wt f_1,\dots,\wt f_m) = \dD^\cD(\wt f_1,\dots,\wt f_m) = \cd(A_1 \times A_2\times \pi^{m-2}).
\]
\end{theorem}

\begin{proof}
Note that $(\wt f_1,\wt f_2,\wt f_3,\dots,\wt f_m)=f_1\times f_2\times\id_{Y^{m-2}}$ by definition, which implies that the pullback over $(\wt f_1,\dots,\wt f_m)$ is the same as that over $f_1\times f_2\times\id_{Y^{m-2}}$. By Lemma~\ref{lem:catX1X2}, the pullback $P$ is contractible. Therefore, 
\[
\sd\secat(P \xrightarrow{\tau} X_1\times X_2\times Y^{m-2})=\sd\cat(X_1\times X_2\times Y^{m-2})=\sd\D(\wt f_1,\dots,\wt f_m),
\]
where the second equality holds by definition. If $Y$ is aspherical, then we have $P \simeq \wh Q$ (compatibly with maps $\tau$ and $\hat q\colon \wh Q\to X_1\times X_2\times Y^{m-2}$ as in~\eqref{diag: second main pullback}), so $\sd\secat(\tau)=\sd\secat(\hat q)$. In other words,
$\sd\D(\wt f_1,\dots,\wt f_m) = \sd\D^\cD(\wt f_1,\dots,\wt f_m)$. Next, if $X_i=K(A_i,1)$, then $X_1 \times X_2 \times Y^{m-2}= K(A_1\times A_2\times \pi^{m-2},1)$ and so by the classical theorem of Eilenberg and Ganea,
\[
\cat(K(A_1\times A_2\times \pi^{m-2},1))=\cd(A_1 \times A_2\times \pi^{m-2}).
\]
Finally, if these groups are torsion-free, then ~\cite{KW} (see also~\cite{Dr,JO}) gives 
\[
\dcat(K(A_1\times A_2\times \pi^{m-2},1))=\cd(A_1 \times A_2\times \pi^{m-2}).
\]
\end{proof}

\begin{ex}
    If $F\xhookrightarrow{i}E\xrightarrow{p}B$ is a fibration and $s\colon B\to E$ is a section of $p$, then 
    \[
    \sd\D(\wt s,\wt i,\id,\dots,\id) = \sd\cat(B\times F\times E^{m-2}),
    \]
    where $\id$ is the identity map on the product $B \times F\times E^{m-2}$. This is because of Theorem~\ref{thm:mapdist}, which applies since $s$ and $i$ satisfy the hypotheses of Lemma~\ref{lem:catX1X2}.
\end{ex}

\section{Behavior under compositions}\label{sec:compositions}
In this section, we study standard properties of the distances $\sd\D$ and $\sd\D^{\cD}$ under compositions to explain their homotopy invariance. Such results for the original invariant $\D$ were obtained in~\cite{MM,BV} using its definition in terms of open sets. This approach does not quite work for $\dD$. So, in this section, we prove such results for the new invariants $\dD$ and $\sd\D^{\cD}$. Note, however, that while we give direct proofs here, Proposition~\ref{prop:classmap} provides an alternative approach in terms of the original results for $\D$.

\begin{proposition}\label{prop:precomp}
Given $f_i\colon X \to Y$ for $1\le i \le m$ and $h\colon W \to X$, we have
\[
\sd\D(f_1\circ h,\dots,f_m\circ h) \leq \sd\D(f_1,\dots,f_m), \quad \text{and}
\]
\[
\sd\D^{\cD}(f_1\circ h,\dots,f_m\circ h) \leq \sd\D^{\cD}(f_1,\dots,f_m).
\]
\end{proposition}

\begin{proof}
Consider the following commutative diagram,
$$\xymatrix@C=3pc{
\wh Q_W \ar[r]^-{r} \ar[d]_-{\wh q^W} & \wh Q_X \ar[r] \ar[d]_-{\wh q^X} & \wh Y^m 
\ar[d]_-{\wh\pi^Y_m} \\
W \ar[r]^-h & X \ar[r]^-{(f_1,\dots,f_m)} & Y^m, 
}
$$
where the rectangle and the right square are the pullbacks from the definitions of $\sd\D^{\cD}(f_1\circ h,\dots,f_m\circ h)$ and $\sd\D^{\cD}(f_1,\dots,f_m)$, respectively, and where $r$ is the obvious whisker map. In particular, they are homotopy pullbacks because $\wh\pi^Y_m$ is a fibration. Then the left square is also a homotopy pullback, see~\cite[Lemma~14]{Mather}. 
It then follows from the standard properties and homotopy invariance of (distributional) sectional category (see~\cite{Sch} and~\cite[Lemma~2.3]{JO}) that
\[
\sd\D^{\cD}(f_1\circ h,\dots,f_m\circ h)=\sd\secat(\wh q^W)\le\sd\secat(\wh q^X)=\sd\D^{\cD}(f_1,\dots,f_m).
\]
For $\sd\D$, we proceed similarly by replacing $\wh Q^A$, $\wh q^A$, and $\wh\pi^Y_m$ with $Q^A$, $q^A$, and $\pi^Y_m$, respectively, for $A\in\{X,W\}$ in the setting of~\eqref{diag: main pullback}.
\end{proof}

\begin{proposition}\label{prop:postcomp}
Given $f_i\colon X \to Y$ for $1\le i \le m$ and $k\colon Y \to Z$, we have
\[
\sd\D(k\circ f_1,\dots,k\circ f_m) \leq \sd\D(f_1,\dots,f_m), \quad \text{and}
\]
\[
\D^{\cD}(k\circ f_1,\dots,k\circ f_m) \leq \D^{\cD}(f_1,\dots,f_m)
\]
If, additionally, the induced map $k_{\#}\colon \pi_1(Y)\to \pi_1(Z)$ is an isomorphism, then the above inequality is saturated, and we also have
\[
\dD^{\cD}(k\circ f_1,\dots,k\circ f_m) \leq \dD^{\cD}(f_1,\dots,f_m).
\]
\end{proposition}

\begin{proof}
First, for $\sd\D$, note that we have the following commutative cube,
\begin{equation}\label{diag: second cube}
\xymatrix@C=3pc{
Q \ar[dr]_-{q^X} \ar[dd]_-{\psi} \ar[rr] & & Y^I \ar[dd]^(.6){\ov k} \ar[dr]^-{\pi^Y_m} &\\
& X \ar[rr]^(.3){(f_1,\dots,f_m)} \ar[dd]_(.3){\id_X} & & Y^m \ar[dd]^-{k^m} \\
P \ar[rr] \ar[dr]_-{p}\ & & Z^I \ar[dr]^-{\pi^Z_m} \\
& X \ar[rr]^-{(k\circ f_1,\dots,k\circ f_m)} && Z^m,
}
\end{equation}
where the squares on the top and bottom are pullbacks, $\ov k$ is induced by $k$ naturally, and $\psi\colon Q\to P$ is the whisker map given by the universal property of the pullback at the bottom. As explained in the proof of Proposition~\ref{prop: cd-D vs D}, we can then easily show using $\psi$ that
\[
\sd\D(k\circ f_1,\dots,k\circ f_m) = \sd\secat(p) \le \sd\secat(q^X) = \sd\D(f_1,\dots,f_m).
\]
For $\D^{\cD}$, there is a more direct proof. If $U \subset X$ is an open set such that 
\begin{equation}\label{eq: conjugate1}
(f_1,\dots,f_m)_\#(\pi_1(U))=((f_1)_{\#},\dots,(f_m)_{\#})(\pi_1(U))
\end{equation}
takes values in a subgroup conjugate to
the diagonal $\Delta_Y$ in $\pi_1(Y)^m$, then 
\begin{equation}\label{eq: conjugate2}
(k\circ f_1,\dots,k\circ f_m)_\#(\pi_1(U))=(k_{\#}\circ (f_1)_{\#},\dots,k_{\#}\circ (f_m)_{\#})(\pi_1(U))    
\end{equation}
takes values in a subgroup conjugate to
the diagonal $\Delta_Z$ in $\pi_1(Z)^m$. Hence, 
\[
\D^{\cD}(k\circ f_1,\dots,k\circ f_m) \leq \D^{\cD}(f_1,\dots,f_m).
\]
If $k_{\#}$ is an isomorphism and the set in~\eqref{eq: conjugate2} takes values in a subgroup conjugate to $\Delta_Z$, then applying the inverse $(k_{\#})^{-1}$ to $\pi_1(Z)$, we see that the set in~\eqref{eq: conjugate1} takes values in a subgroup conjugate to $\Delta_Y$, thereby giving the reverse inequality
\[
\D^{\cD}(f_1,\dots,f_m)\leq \D^{\cD}(k\circ f_1,\dots,k\circ f_m).
\]
Also, 
the map $\phi'\colon \wt Y^m\to \wt Z^m$ between the universal covers induced by $k$ in turn induces a map $\phi\colon \wh Y^m\to \wh Z^m$ on the $\pi_1(Y)\cong\pi_1(Z)$-quotients. So, we can replace the spaces $Q$, $P$, $Y^I$ and $Z^I$ in~\eqref{diag: second cube} with $\wh Q$, $\wh P$, $\wh Y^m$ and $\wh Z^m$, respectively, and the maps $q^X$, $p$, $\pi^A_m$, and $\ov k$ in~\eqref{diag: second cube} with $\wh q^X$, $\wh p$, $\wh\pi^A_m$, and $\phi$, respectively, for $A\in\{Y,Z\}$. Then proceeding as before, we get a whisker map $\wh\psi\colon\wh Q\to \wh P$ such that $\wh p\circ \wh \psi=\wh q^X$, which helps get
\[
\dD^{\cD}(k\circ f_1,\dots,k\circ f_m) \leq \dD^{\cD}(f_1,\dots,f_m).
\]
\end{proof}

\begin{corollary}\label{cor: homotopy-invariance}
    If $f_i\colon X\to Y$ are maps for $1\le i \le m$, $h\colon W\to X$ has a right homotopy inverse, and $k\colon Y\to Z$ has a left homotopy inverse, then 
    \[
    \sd\D(f_1,\dots,f_m)=\sd\D(k\circ f_1\circ h, \dots, k\circ f_m\circ h), \quad \text{and}
    \]
    \[
    \D^{\cD}(f_1,\dots,f_m)=\D^{\cD}(k\circ f_1\circ h, \dots, k\circ f_m\circ h).
    \]
    If, additionally, $k_\#\colon \pi_1(Y)\to \pi_1(Z)$ is an isomorphism, then
    \[
    \dD^{\cD}(f_1,\dots,f_m)=\dD^{\cD}(k\circ f_1\circ h, \dots, k\circ f_m\circ h).
    \]
    In particular, $\sd\D$ and $\sd\D^{\cD}$ are homotopy invariants.
\end{corollary}
\begin{proof}
    Suppose $h'\colon X\to W$ and $k'\colon Z\to Y$ are such that $h\circ h'\simeq\id_X$ and $k'\circ k\simeq \id_Y$. Then $f_i\simeq k'\circ (k\circ f_i\circ h)\circ h'$ for each $i$, so applying Propositions~\ref{prop:precomp} and~\ref{prop:postcomp}, we get for $\sd\D$, for instance, that
    \[
    \sd\D(f_1,\dots,f_m)\le \sd\D(k\circ f_1\circ h,\dots,k\circ f_m\circ h)\le \sd\D(f_1,\dots,f_m). 
    \]
Similar inequalities are obtained also for $\sd\D^{\cD}$ due to our hypothesis.
\end{proof}

We now see a way to relate the diagonal homotopic distances between maps with the homotopic distances between their post-compositions with classifying maps.

\begin{proposition}\label{prop:classmap}
Let $f_i\colon X\to Y$ be maps between path-connected CW complexes. If $\pi=\pi_1(Y)$ and $h\colon Y \to K(\pi,1)$ is a classifying map, then
\[
\sd\D^\cD(f_1,\dots,f_m) = \sd\D(h\circ f_1,\dots,h\circ f_m).
\]
\end{proposition}
\begin{proof}
Since $h\circ f_i\colon X\to K(\pi,1)$, we have from Proposition~\ref{prop: cd-D vs D} that
\[
\sd\D(h\circ f_1,\dots,h\circ f_m)=\sd\D^{\cD}(h\circ f_1,\dots,h\circ f_m).
\]
But the quantity on the right is equal to $\sd\D^\cD(f_1,\dots,f_m)$ due to Corollary~\ref{cor: homotopy-invariance} since $h_{\#}$ is an isomorphism.
\end{proof}

We close this section with another property of homotopic distances.

\begin{proposition}\label{prop: product of maps}
    If $f_i\colon X\to Y$ and $h\colon Z\to W$ are maps for $1\le i \le m$, then
    \[
    \sd\D(f_1\times h,\dots,f_m\times h)=\sd\D(f_1,\dots,f_m)=\sd\D(h\times f_1,\dots,h\times f_m),
    \]
    where $f_i\times h\colon X\times Z\to Y\times W$ is the product map $(f_i\times h)(x,z)=(f_i(x),h(z))$, and $h\times f_i\colon Z\times X\to W\times Y$ is defined analogously for each $i$.
\end{proposition}

\begin{proof}
    Let $\iota_1\colon X\hookrightarrow X\times Z$ and $\iota_2\colon Z\hookrightarrow X\times Z$ be the inclusions into the first and second factors, respectively, and let $\pr_1\colon Y\times W\to Y$ and $\pr_2\colon Y\times W\to W$ be the obvious projections.
    Then $f_i=\pr_1\circ (f_i\times h)\circ\iota_1=\pr_2\circ (h\times f_i)\circ\iota_2$. So, Propositions~\ref{prop:precomp} and~\ref{prop:postcomp} imply
    \[
    \sd\D(f_1,\dots,f_m)\le\min\{\sd\D(f_1\times h,\dots,f_m\times h),\  \sd\D(h\times f_1,\dots,h\times f_m)\}.
    \]
    For the converse, 
    let $\sd\D(f_1,\dots,f_m)=n$. There exists a map $s\colon X\to (Y^I)_n$, where for each $x\in X$ and $\phi\in\supp(s(x))$, we have $\phi(t_i)=f_i(x)$ for $1\le i \le m$. Let $c\colon W\to W^I$ be the map such that for each $w\in W$, we have $c(w)(t)=w$ for all $t\in [0,1]$. Also, the inclusion $\sigma\colon (Y^I)_n\times W^I\to (Y^I\times W^I)_n\cong((Y\times W)^I)_n$ is defined as 
    \[
\sigma\left(\sum\lambda_\phi\phi,\psi\right)=\sum \lambda_\phi\left(\phi,\psi\right).
    \]
    Take the composition $\tau:=\sigma\circ (\id\times c)\circ (s\times h)\colon X\times Z\to ((Y\times W)^I)_n$, where $\id$ is the identity map on $(Y^I)_n$. If $(x,z)\in X\times Z$ and $\rho\in\supp(\tau(x,z))$, then $\rho(t_i)=(f_i(x),h(z))=(f_i\times h)(x,z)$ for $1\le i \le m$. It follows by definition  that $\dD(f_1\times h,\dots,f_m\times h)\le n=\dD(f_1,\dots,f_m)$. The proof of the parallel inequality $\dD(h\times f_1,\dots,h\times f_m)\le n$ is symmetric. Finally, we note that the proofs of
    \[
    \max\{\D(f_1\times h,\dots,f_m\times h),\  \D(h\times f_1,\dots,h\times f_m)\}\le\D(f_1,\dots,f_m)
    \]
    are identical to the analogous corresponding proofs for $\dD$: one uses the iterated fiberwise join constructions there instead of the distributional constructions.
\end{proof}

\section{Cohomological lower bounds}\label{sec:cohombounds}

In this section, we obtain lower bounds to the $m$-th sequential (distributional) homotopic distance between maps in terms of the cohomology of the maps involved. 

Let us fix an arbitrary integer $m\ge 2$ and maps $f_i\colon X\to Y$ for $1\le i \le m$. For $k\ge 1$, let $\delta_k\colon Y\to SP^k(Y)$ be the diagonal map. Here, $SP^k(Y)$ is the \emph{$k$-th symmetric product of $Y$}, which is the orbit space of the permutation action of the symmetric group $\Sigma_k$ on the Cartesian product $Y^k$. It is convenient to write any element $[y_1,\dots,y_k]\in SP^k(Y)$ as a formal sum as follows: $\sum n_iy_i$, where $n_i\ge 1$ is an integer, $\sum n_i=k$, and $y_i\in Y$, subject to the equivalence $\ell_1y+\ell_2y=(\ell_1+\ell_2)y$. In particular, $\delta_k(y)=ky$ for each $y\in Y$. 

\begin{lemma}\label{lem: cover}
   If $\dD(f_1,\dots,f_m)<n$, then $X$ can be covered by $n$ sets $A_j$ such that for each $j$, $\delta_{n!}\circ f_{i|A_j}\simeq \delta_{n!}\circ f_{k|A_j}$ for $1\le i,k\le m$.
\end{lemma}

\begin{proof}
Since $\dD(f_1,\dots,f_m)<n$, there exists a map $s\colon X\to (Y^I)_n$ such that, given $x\in X$ and $\phi\in\supp(s(x))$, we have $\phi(t_i)=f_i(x)$ for each $i$. 
For $1\le j \le n$, define 
\[
A_j:=\left\{x\in X \ \middle|\ \left|\supp(s(x))\right|=j\right\}.
\]
Next, for each $i\in \{1,\dots,m-1\}$, define a map $H_{i,j}\colon A_j\times I\to SP^{n!}(Y)$ as follows:
\[
H_{i,j}(x,t)=\sum_{\phi\in\supp(s(x))}\frac{n!}{j}\ \phi(t(t_{i+1}-t_i)+t_i).
\]
Clearly, $H_{i,j}$ is a homotopy from $\delta_{n!}\circ f_{i|A_j}$ to $\delta_{n!}\circ f_{i+1|A_j}$ for $1\le i \le m-1$.
\end{proof}

Since the sets $A_i$ from Lemma~\ref{lem: cover} need not be open or closed, we present the lower bounds to $\dD(f_1,\dots,f_m)$ in Alexander--Spanier cohomology. We put $\ast$ in the superscript to denote maps induced in this cohomology.

\begin{theorem}\label{thm: lower bound}
    Let $R$ be a ring, $f_i\colon X\to Y$ be maps, and  $c_i\ge 1$ be integers for $1\le i\le m$ such that $\sum_{i=1}^mc_i=0$. If  $n\ge 1$ is such that 
    \[
\cu_R\left(\Im\left(\sum_{i=1}^mc_i\left(\delta_{n!}\circ f_i\right)^*\right)\right)\ge n,
    \]
    then $\dD(f_1,\dots,f_m)\ge n$.
\end{theorem}

Here, for the subset $S$ of a cohomology ring with coefficients in $R$, $\cu_R(S)$ denotes the \emph{cup length} of $S$, i.e., the length of the longest non-trivial cup product(s) of positive-degree Alexander--Spanier cohomology classes in $S$. We note that $\Im(\sum c_i(\delta_{n!}\circ f_i)^*)$ may not be a ring, but its $\cu_R$ is still defined.

\begin{proof}
By our assumption, there exist classes $\alpha_j\in H^{k_j}(X;R)$ for $1\le j \le n$ such that $k_j\ge 1$, $\alpha_j=\sum c_i(\delta_{n!}\circ f_i)^*(\beta_j)$ for some $\beta_j\in H^{k_j}(SP^{n!}(Y);R)$, and $\alpha_1\cdots\alpha_n\ne 0$. Now, suppose $\dD(f_1,\dots,f_m)< n$, so that we have sets $A_j$ as in Lemma~\ref{lem: cover} for $1\le j \le n$. Fix any $j$ and consider the following commutative diagram, where the bottom row comes from the Alexander--Spanier cohomology long exact sequence corresponding to the pair $(X,A_j)$:
\[
\xymatrix@R=2pc@C=1.5pc{
& & H^{k_j}(SP^{n!}(Y);R) \ar[dr]^-{\theta_j} \ar[d]_-{\sum_{i=1}^mc_i(\delta_{n!}\circ f_i)^*} & &
\\
\cdots\ar[r]^-{}  & H^{k_j}(X,A_j;R) \ar[r]_-{\phi_j^*} & H^{k_j}(X;R)\ar[r]_-{\psi_j^*}& 
H^{k_j}(A_j;R) \ar[r]^-{} & \cdots.
}
\]
Here, $\phi_j\colon X\hookrightarrow (X,A_j)$ and $\psi_j\colon A_j\hookrightarrow X$ are the obvious inclusion maps, and $\theta_j:=\sum c_i(\delta_{n!}\circ f_i\circ\psi_j)^*$. Since $\delta_{n!}\circ f_i\circ \psi_j\simeq \delta_{n!}\circ f_{i+1}\circ \psi_j$ for $1\le i\le m-1$ (\emph{cf}. Lemma~\ref{lem: cover}) and $\sum c_i=0$, the map $\theta_j$ is trivial. In particular, 
\[
\psi_j^*(\alpha_j)=\psi_j^*\left( \sum_{i=1}^m c_i(\delta_{n!}\circ f_i)^*(\beta_j)\right)=\theta_j(\beta_j)=0.
\]
By exactness, there exists $\ov\alpha_j\in H^{k_j}(X,A_j;R)$ such that $\phi_j^*(\ov\alpha_j)=\alpha_j$. Writing $k=\sum_{j=1}^nk_j$, we get by the naturality of the cup product that
\[
0\cong H^k(X,X;R)\cong H^k\left(X,\bigcup_{j=1}^nA_j;R\right)\ni \ov\alpha_1\cdots\ov\alpha_n\ \longmapsto\ \alpha_1\cdots\alpha_n\ne 0,
\]
which is a contradiction. Therefore, $\dD(f_1,\dots,f_m)\ge n$.
\end{proof}

\begin{remark}\label{rem: dist lower bound}
Now, note that if $\D(f_1,\dots,f_m)<n$ for the original distance, then by definition, 
we have an \emph{open} cover $\{B_j\}_{j=1}^n$ of $X$ such that 
$f_{i|B_j}\simeq f_{k|B_j}$ for all $1\le i,k\le m$ and each $j$.
The proof of Theorem~\ref{thm: lower bound} then immediately yields the following 
lower bound to $\D(f_1,\dots,f_m)$ in singular cohomology:
\[
\D(f_1,\dots,f_m)\ge \cu_R\left(\Im\left(\sum_{i=1}^mc_if_i^*\right)\right),
\]
where again we take $\sum c_i=0$.
In~\cite{MM}, the above was obtained in the special case $m=2$ with $c_1=-c_2$.
\end{remark}
\begin{corollary}\label{cor: same lower bound}
    If $f_i\colon X\to Y$ are maps, $c_i$ are integers for $1\le i \le m$ such that $\sum_{i=1}^mc_i=0$, and $Y$ has the homotopy type of a finite CW complex, then in Alexander--Spanier cohomology, we have that
\[
\dD(f_1,\dots,f_m)\ge \cu_{\Q}\left(\Im\left(\sum_{i=1}^mc_if_i^*\right)\right).
\]
\end{corollary}
\begin{proof}
For brevity, write 
\[
S:=\Im\left(\sum_{i=1}^m c_i f_i^*\right) \ \text{ and } \ S_k:=\Im\left(\sum_{i=1}^m c_i(\delta_{k}\circ f_i)^*\right)
\]
for each $k\ge 1$. Let $\alpha\in S\subset H^*(X;\Q)$, so that there exists $\beta\in H^*(Y;\Q)$ satisfying $\sum c_if_i^*(\beta)=\alpha$. Because of the assumption on $Y$, for each $k,j\ge 1$, the map $\delta_k^*\colon H^j(SP^k(Y);\Q)\to H^j(Y;\Q)$ is an epimorphism, see~\cite[Subsection~4.1]{DJ}. We can use this result here because Alexander--Spanier cohomology coincides with singular cohomology on (locally) finite CW complexes. In particular, there exists $\ov\beta_k\in H^*(SP^k(Y);\Q)$ such that $\delta_k^*(\ov\beta_k)=\beta$. This implies that $\alpha\in S_k$. It follows that $\cu_{\Q}(S)=\cu_{\Q}(S_k)$ for each $k$, so we are done in view of Theorem~\ref{thm: lower bound}.   
\end{proof}

Let $f\colon X\to Y$ be a map and $\pr_i\colon X^m\to X$ be the projection for each $i$. Taking $f_i=f\circ \pr_i\colon X^m\to Y$, we recover the known lower bounds to $\sd\TC_m(f)$ in view of Proposition~\ref{prop: recovering tcf}.

\begin{ex}
    Let $X=Y=U(2)$ be the compact Lie group of unitary $2\times 2$ matrices, $f=\id_X$, and $g\colon X\to X$ be the map $g(A)=A^*$, where $A^*$ is the conjugate transpose of $A\in X$. It is easy to show that $\cu_{\Q}(\Im(f^*-g^*))=2=\D(f,g)$, see~\cite[Example~5.4]{MM}. By Corollary~\ref{cor: same lower bound}, we then get $\dD(f,g)=2=\D(f,g)$.
\end{ex}

We will see more general instances of the sharpness of our cohomological lower bounds to $\sd\D$ in Section~\ref{sec:mapgroup}.

\section{Maps to $H$-spaces}\label{sec: h-spaces}

In this section, we consider maps to nice $H$-spaces and express the (diagonal) homotopic distances between them in terms of the (one-)categories of certain associated maps. We first recall some definitions.

An \emph{$H$-space} $G$ is a topological space that admits a \emph{multiplication map} $\mu\colon G\times G\to G$ and an \emph{identity element} $e\in G$ such that $\mu_{|G\times\{e\}}\simeq \id_G\simeq\mu_{|\{e\}\times G}$. Unless stated otherwise, $H$-spaces considered in this paper are \emph{not} homotopy-associative.

A \emph{group-like space} $G$ (also called an \emph{$H$-group}) is a 
topological space that is a 
homotopy-associative $H$-space with a multiplication map $\mu\colon G\times G\to G$, an identity element $e\in G$, and 
a compatible \emph{inversion map} $\eta\colon G\to G$. By homotopy-associativity, we mean $\mu\circ(\mu\times \id_G)\simeq \mu\circ (\id_G\times\mu)$, and by compatibility of the inversion, we mean 
$\mu\circ(\id_G\times\eta)\circ\Delta\simeq c_e\simeq 
\mu\circ(\eta\times \id_G)\circ\Delta$, where 
$\Delta\colon G\to G\times G$ is the diagonal map and 
$c_e\colon G\to G$ is the constant map to the identity $e\in G$. 

In this section, we consider path-connected CW $H$-spaces that are not necessarily homotopy-associative and hence not necessarily group-like (for example, $S^7$), and path-connected  group-like spaces that are not necessarily homotopic to CW complexes (for example, infinite-dimensional Hilbert spaces). For
us, the term ``group-like'' will include path-connected, but we will often emphasize this in
our statements.

\begin{remark}\label{rem: map F}
Let $G$ be a path-connected group-like space or a path-connected  CW $H$-space with a multiplication $\mu$, and let $\pr_i\colon G\times G\to G$ be the projections for $i\in\{1,2\}$. Then there exists a map $F\colon G\times G\to G$ such that 
\[
\mu\circ (F,\pr_2)\simeq \pr_1.
\]
Indeed, if $G$ is group-like, then the composition $\mu\circ (\id_G\times\eta)$ has this property because of the homotopy associativity of $\mu$. On the other hand, if $G$ is a CW $H$-space, then such $F$ is the \emph{unique} solution to the equation $[\pr_2]\bullet[F]=[\pr_1]$, where $\bullet$ is the ``product'' induced by $\mu$ on the set $[G\times G,G]$ of homotopy classes of maps $G\times G\to G$. This is because $[G\times G,G]$ is an algebraic loop due to James, see ~\cite{James loop}. Also, in the latter case, we still have $F\circ\Delta\simeq c_e$. Indeed, $[\pr_2\circ\Delta]\bullet[F\circ \Delta]=[\pr_1\circ\Delta]$ and $[\pr_2\circ\Delta]\bullet[c_e]=[\pr_1\circ\Delta]$ are two solutions to the same equation, so the uniqueness of solutions to equations~\cite{James loop} in the algebraic loop $[G,G]$ implies $[F\circ\Delta]=[c_e]$. 
\end{remark}

Moving forward, we take $F=\mu\circ (\id_G\times\eta)$ if $G$ is a path-connected group-like space. Otherwise, if $G$ is a path-connected  CW $H$-space, we take $F\colon G\times G\to G$ to be as in Remark~\ref{rem: map F}. 

Given $m\ge 2$, we can define in any case a map $F_m\colon G^m\to G^{m-1}$ as 
\begin{equation}\label{eq: map Fm}
F_m(x_1,x_2,\dots,x_m)=(F(x_1,x_2),F(x_2,x_3),\dots, F(x_{m-1},x_m)).
\end{equation}

\subsection{Homotopic distance and LS-category}\label{subsec: dist as cat}
We prove the first main result of this section. Throughout, we follow the above-mentioned conventions.

\begin{theorem}\label{thm: main1}
    Let $G$ be a path-connected group-like space or a path-connected CW $H$-space, and $m\ge 2$ be an integer. If $f_i\colon X\to G$ are maps for $1\le i \le m$, then 
    \[
    \sd\D(f_1,\dots,f_m)=\sd\cat(F_m\circ (f_1,\dots,f_m)),
    \]
    where $(f_1,\dots,f_m)\colon X\to G^m$ is the map $(f_1,\dots,f_m)(x)=
(f_1(x),\dots,f_m(x))$, and $F_m\colon G^m\to G^{m-1}$ is as in~\eqref{eq: map Fm}.
\end{theorem}

\begin{proof}
Let $P_0(G)$ be the space of paths terminating at the identity $e\in G$, and let $p_G\colon P_0(G)\to G$ be the evaluation-at-$0$ fibration. The space $P$ corresponding to the pullback of the map $F_m\colon G^m\to G^{m-1}$ along the $(m-1)$-th power fibration $p_G^{m-1}\colon P_0(G)^{m-1}\to G^{m-1}$ is
\begin{multline*}
P=\bigl\{\left(x_1,\dots,x_m,\alpha_1,\dots,\alpha_{m-1}\right)\in G^m\times P_0(G)^{m-1}\ \big|\ \alpha_i(1)=e \ \text{and}
\\
\alpha_i(0)=F(x_i,x_{i+1}) \ \text{for}\ 1\le i \le m-1 \bigr\},
\end{multline*}
and the pullback map $p:=F^*_mp^{m-1}_G$ is the projection fibration $p\colon P\to G^m$. 
(Note that there is an obvious homotopy equivalence over $G^{m-1}$ given by $P_0(G^{m-1}) \to 
P_0(G)^{m-1}$. This then is a fibre homotopy equivalence and all our invariants may be
defined using $P_0(G)^{m-1}$.)

For a given element $(x_1,\dots,x_m,\alpha_1,\dots,\alpha_{m-1})\in P$, define a path $\wt\alpha_i\in G^I$ for $1\le i \le m-1$ by 
\[
\wt\alpha_i(t)=\mu(\alpha_i(t),x_{i+1}).
\]
The homotopy $H\colon G\times G\to G^I$ from the map $\pr_1$ to $\mu\circ (F,\pr_2)$, and the 
homotopy $K\colon G\to G^I$ from $\mu_{\{e\}\times G}$ to $\id_G$ then show that 
$H(x_i,x_{i+1})$ is a path from $x_i$ to $\wt\alpha_i(0)$, and $K(x_{i+1})$ is a path 
from $\wt\alpha_i(1)$ to $x_{i+1}$. Thus, the concatenation $\wh\alpha_i:=H(x_i,x_{i+1})\cdot \wt\alpha_i\cdot K(x_{i+1})$ is a path from $x_i$ to $x_{i+1}$ for each $i$. Now, define the concatenation
\[
\wh\alpha:=\wh\alpha_1\star\wh\alpha_2\star\cdots\star\wh\alpha_{m-1},
\]
from Subsection~\ref{subsec: concat}, where we set $a_i=\tfrac{1}{t_{i+1}}$ for 
$1\le i\le m-1$. Then the map $\theta\colon P\to G^I$ given by
\[
\theta\left(x_1,\dots,x_m,\alpha_1,\dots,\alpha_{m-1}\right)=\wh\alpha
\]
satisfies $\pi^G_m\circ\theta= p$, where $\pi^G_m\colon G^I\to G^m$ evaluates each path in $G$ at timestamps $t_i$ for $1\le i\le m$. 

Next, for a given $\gamma\in G^I$, define a path $\gamma_i\in G^I$ for $1\le i \le m-1$ such that 
\[
\gamma_i(t)=F(\gamma(t(t_{i+1}-t_i)+t_i),\gamma(t_{i+1})).
\]
It follows from Remark \ref{rem: map F} that
there is a homotopy $L\colon G\to G^I$ from $F\circ\Delta$ to $c_e$. Indeed, this is obvious if $G$ is a group-like space since then we have $F=\mu\circ (\id_G\times\eta)$ by our initial choice, and this is explained in Remark~\ref{rem: map F} for the case of CW $H$-spaces. Thus, the concatenation 
$\gamma_i':=\gamma_i\cdot L(\gamma(t_{i+1}))$ is a path from $F(\gamma(t_i),\gamma(t_{i+1}))$ to $e$. Then the map $\phi\colon G^I\to P$ given by
\[
\phi(\gamma)=\left(\gamma(t_1),\dots,\gamma(t_m),\gamma_1',\dots,\gamma'_{m-1}\right)
\]
is well-defined and satisfies $p\circ\phi=\pi^G_m$. Therefore, we have the following commutative diagram, where the squares on the left and the right are pullbacks:
\begin{equation*}\label{diag: big diag}
    \xymatrix@C=4pc{
Q \ar[r]^-{b} \ar[d]_-{q^X} & G^I \ar@<.5ex>[r]^-{\phi} \ar[d]^-{\pi^G_m}  & P \ar@<.5ex>[l]^-{\theta} \ar[d]^-{p} \ar[r]^-{a} & P_0(G)^{m-1} \ar[d]^-{p_G^{m-1}}
\\
X \ar[r]^-{(f_1,\dots,f_m)} & G^m \ar@{=}[r] & G^m \ar[r]^-{F_m} & G^{m-1}.
    }
\end{equation*}
Note that the iterated fiberwise join and distributional constructions on the left and right squares above induce their respective (homotopy) pullbacks, and such constructions on the middle square induces a (homotopy) commutative square~\cite{Sch,JO}, i.e., $p_k\simeq F_m^*(p_G^{m-1})_k$, $(q^X)_k\simeq (f_1,\dots,f_m)^*(\pi^G_m)_k$, $p_k\circ \phi_k\simeq(\pi^G_m)_k$, and $(\pi^G_m)_k\circ\theta_k\simeq p_k$ for each $k$
for the distributional maps, and similarly, for the analogous fiberwise joined maps. We do the rest of the proof for $\dsecat$, noting that the exact same idea will work for $\secat$, using the iterated 
join construction instead of the distributional construction (so in the proven inequalities below, 
we write $\sd\secat$).

If $r\colon X\to (P_0(G)^{m-1})_k$ is a lift of $F_m\circ (f_1,\dots,f_m)$ along $(p_G^{m-1})_k$, then by the universal property of homotopy pullbacks, there exists a whisker map $t\colon X\to P_k$ such that $p_k\circ t\simeq (f_1,\dots,f_m)$. Taking $\theta_k\circ t\colon X\to (G^I)_k$, we get a whisker map $\sigma\colon X\to Q_k$ such that $(q^X)_k\circ\sigma\simeq\id_X$. Indeed, here we used the fact that $(\pi^G_m)_k\circ\theta_k\simeq p_k$. Since $(q^X)_k$ is a fibration, we get a true section. Therefore, 
\begin{align*}
\sd\D(f_1,\dots,f_m)=\sd\secat(q^X) & \le \sd\secat((F_m\circ (f_1,\dots,f_m))^*p_G^{m-1})
\\
& =\sd\cat(F_m\circ (f_1,\dots,f_m)).
\end{align*}
Conversely, if $s\colon X\to Q_k$ is a section of $(q^X)_k$, then 
\[
a_k\circ \phi_k\circ b_k\circ s\colon X \to (P_0(G)^{m-1})_k
\]
is a homotopy lift of $F_m\circ (f_1,\dots,f_m)$ along $(p_G^{m-1})_k$, which yields a true lift since the latter is a fibration. Therefore, 
\[
\sd\secat((F_m\circ (f_1,\dots,f_m))^*p_G^{m-1})\le \sd\secat(q^X).
\]
Hence, we have $\sd\D(f_1,\dots,f_m)=\sd\cat(F_m\circ (f_1,\dots,f_m))$.
\end{proof}

Theorem~\ref{thm: main1} is a vast generalization of several well-known results and computations from the literature. Before specifying some of them, we recall fundamental properties of $\cat(f)$ from~\cite{CLOT} and their natural distributional analogues.

\begin{proposition}\label{prop:catfprops}
For a map $f\colon X \to Y$, we have that
\begin{enumerate}
\item $\sd\cat(f) \leq \min\{\sd\cat(X),\, \sd\cat(Y)\}$;
\item If $g\colon Y \to Z$, then $\sd\cat(g\circ f) \leq \min\{\sd\cat(f),\,\sd\cat(g)\}$;
\item If $f \simeq g$, then $\sd\cat(f)=\sd\cat(g)$;
\item If $f$ is a homotopy retraction, then $\sd\cat(f)=\sd\cat(Y)$.
\end{enumerate}
\end{proposition} 
Let us now make the following observations.
\begin{enumerate}[(i)]
\item In the special case $X=G\times G$, it was shown in~\cite[Theorem~4.1]{MM} that $\D(f_1,f_2)\le\cat(G)$. This is a direct corollary to Theorem~\ref{thm: main1} in the case $m=2$ since $F_2\circ (f_1,f_2)\colon G\times G\to G$ in view of Proposition~\ref{prop:catfprops} (1).

\item It follows from the proofs in~\cite{LS,Ja1} that $\sd\TC_m(G)=\sd\cat(F_m)$. This is the special case $X=G^m$ and $f_i=\pr_i\colon G^m\to G$ of Theorem~\ref{thm: main1}  in view of Proposition~\ref{prop: recovering tcf}.

\item If $G$ is group-like, then taking $\delta=\mu\circ(\eta\times\id_G)$, we recover the computation $\D(\id_G,\eta)=\D(\mu,\delta)$ from~\cite{MM} in a somewhat different way. Indeed, for $F_2=F:=\mu\circ(\id_G\times\eta)$, we get $F\circ (\id_G,\eta)(x)\simeq \mu(x,x)\simeq F\circ (\mu,\delta)(x,y)$ for all $x,y\in G$, and thus the following homotopy commutative diagram:
\begin{equation}\label{diag: antiparallel diagonal arrows}
\xymatrix{
G\times G 
\ar[rr]^-{F\circ (\mu, \delta)} \ar@<.5ex>[dr]^-{\pr_1} && G \\
& G, \ar@<.5ex>[ul]^-{i_1} \ar[ur]_-{F\circ (\id_G, \eta)} 
}    
\end{equation}
where $i_1(x)=(x,e)$. Of course, (2) and (3) in Proposition~\ref{prop:catfprops} give
\[
\sd\cat(F\circ (\mu, \delta))\le \sd\cat(F\circ (\id_G, \eta))\le \sd\cat(F\circ (\mu, \delta)).
\]
But now we are done due to Theorem~\ref{thm: main1} as the quantity in the middle is $\sd\D(\id_G, \eta)$ and the other quantity is $\sd\D(\mu, \delta)$.

\item For $k\ge 2$, define $\mu_k\colon G\to G$ as the composition of $\mu\circ\Delta\colon G\to G$ with itself $(k-1)$-times, and set $\mu_0=c_e$ and $\mu_1=\id_G$. If $G$ is group-like, we can similarly recover the computation $\D(\mu_{k-\ell},c_e)=\D(\mu_k,\mu_{\ell})$ from~\cite{MM} for $k\ge\ell$. Indeed, it is not difficult to see that for any $x\in G$, we have
\[
F\circ (\mu_k,\mu_{\ell})(x)=\mu(\mu_k(x),\eta(\mu_\ell(x)))\simeq \mu_{k-\ell}(x)\simeq \mu(\mu_{k-\ell}(x),\eta(e))=F\circ (\mu_{k-\ell},c_e)(x).
\]
So, Theorem~\ref{thm: main1} and Proposition~\ref{prop:catfprops} (3) give the equalities
\[
\sd\D(\mu_k,\mu_\ell)=\sd\cat(F\circ (\mu_k,\mu_{\ell}))=\sd\cat(F\circ (\mu_{k-\ell},c_e))=\sd\D(\mu_{k-\ell},c_e).
\]
\end{enumerate}

We conclude this subsection with the following new example.

\begin{ex}\label{exam:conjdist}
For a group-like space $G$ and an element $g\in G$, let $f_g\colon G \to G$ be defined as $f_g(h)=F(\mu(g,h),g)=\mu(\mu(g,h),\eta(g))$.
For any $k\in G$, if $c_{k}\colon G\to G$ denotes the constant map to $k$, then we claim that
\[
\sd\D(f_g,c_{\eta(g)})=\sd\cat(G)=\sd\D(\id_G,c_e).
\]
Note that since $F=\mu\circ (\id_G\times\eta)$, we have for any $h\in G$ that
\[
F\circ (f_g,c_{\eta(g)})(h)=\mu(\mu(\mu(g,h),\eta(g)),\eta(\eta(g)))\simeq \mu(g,h)=L_g(h),
\]
where the relation holds because of the properties of the homotopy-associative multiplication $\mu$, and where $L_g\colon G\to G$ is the multiplication map $L_g(h)=\mu(g,h)$, which is a homotopy equivalence. Therefore, by Theorem~\ref{thm: main1}, parts (3) and (4) in Proposition~\ref{prop:catfprops}, and Proposition~\ref{prop: recovering tcf}, we have that
\[
\sd\D(f_g,c_{\eta(g)})=\sd\cat(F\circ (f_g,c_{\eta(g)}))=\sd\cat(L_g)=\sd\cat(G)=\sd\D(\id_G,c_e).
\]
\end{ex}

\subsection{$\cD$-homotopic distance and one-category}
We now prove the second main result of this section, which is analogous to Theorem~\ref{thm: main1}. We first note the following.

\begin{remark}\label{rem: new map F}
If $G$ is a path-connected CW $H$-space with multiplication $\mu$ and identity $e$, then the algebraic loop structure on $[G,G]$ due to James~\cite{James loop} shows that there exists an ``inversion'' map $\eta'\colon G\to G$ such that 
\[
\mu\circ (\id_G\times\eta')\circ\Delta\simeq c_e.
\]
Let us set $F'=\mu\circ (\id_G\times\eta')\colon G\times G\to G$. We recall from Remark~\ref{rem: map F} that such a map $F'=F=\mu\circ(\id_G\times\eta)$ exists if $G$ is group-like (without necessarily having the homotopy type of a CW complex) with its inversion $\eta$. 
\end{remark}

As in~\eqref{eq: map Fm}, we can define for each $m\ge 2$ a map $F'_m\colon G^m\to G^{m-1}$ as follows using the map $F'\colon G\times G\to G$ from Remark~\ref{rem: new map F}:
\begin{equation}\label{eq: new map Fm}
F'_m(x_1,x_2,\dots,x_m)=(F'(x_1,x_2),F'(x_2,x_3),\dots, F'(x_{m-1},x_m)).
\end{equation}

\begin{theorem}\label{thm: main2}
    Let $G$ be a path-connected  group-like space or a path-connected  CW $H$-space and $m\ge 2$ be an integer. If $f_i\colon X\to G$ are maps for $1\le i \le m$, then 
    \[
    \sd\D^{\cD}(f_1,\dots,f_m)=\sd\cat_1(F'_m\circ (f_1,\dots,f_m)),
    \]
    where $F'_m\colon G^m\to G^{m-1}$ is the map defined in~\eqref{eq: new map Fm}.
\end{theorem}

\begin{proof}
Set $\pi:=\pi_1(G)$. If $F'\colon G\times G\to G$ is a map as in Remark~\ref{rem: new map F}, then note that by its definition, $F'$ induces on first homology the map $F'_{*}(a,b)=a-b$. But since $\pi_1(G)\cong H_1(G)$ for the path-connected  $H$-space $G$, we see that the induced map of fundamental groups is $F'_{\#}(a,b)=a-b$. Therefore, the map $(F_m')_\#\colon \pi^m\to \pi^{m-1}$ induced by the map $F_m'$ from~\eqref{eq: new map Fm} satisfies
\[
(F_m')_\#(a_1,a_2,\dots,a_m)=(a_1-a_2,a_2-a_3,\dots,a_{m-1}-a_m).
\]
Clearly, $\Ker((F_m')_\#)=\Delta\subset \pi^m$, which is the diagonal subgroup. Therefore, we can lift the map $F_m'\colon G^m\to G^{m-1}$ to a map $\wt F'_m\colon \wh G^m\to \wt G^{m-1}$ in the diagram
\[
\xymatrix@C=2pc{
\wh G^m \ar[r]^-{\wt F_m'} \ar[d]_-{\wh\pi^G_m} & \wt G^{m-1} \ar[d]^-{p}
\\
G^m\ar[r]^-{F_m'}& G^{m-1},
}
\]
where $p\colon \wt G^{m-1}\to G^{m-1}$ is the universal covering map. It is easy to check that $\Im((F'_m)_\#)=\pi^m/\Delta\cong\pi^{m-1}$,  and this in turn shows that $\wt F'_m$ restricted to each fiber $\pi^m/\Delta\cong\pi^{m-1}$ is a bijection. This then implies that the diagram above
is a (homotopy) pullback. We append the above pullback from the left using the 
(homotopy) pullback from~\eqref{diag: second main pullback} to get the following commutative diagram
\[
\xymatrix@C=2cm{
\wh Q \ar[r]^-{} \ar[d]_-{\wh q^X} & \wh G^m \ar[r]^-{\wt F_m'} \ar[d]_-{\wh\pi^G_m} & \wt G^{m-1} \ar[d]^-{p}
\\
X \ar[r]^-{(f_1,\dots,f_m)} & G^m\ar[r]^-{F_m'}& G^{m-1},
} 
\]
where the large rectangle is a homotopy pullback, see~\cite{Mather}. By definition (\emph{cf.} Lemma~\ref{lem: cD-dist as secat}), we have $\sd\D^{\cD}(f_1,\dots,f_m)=\sd\secat(\wh q^X)$. Also, the above homotopy pullback gives
\[
\sd\secat(\wh q^X) = \sd\secat((f_1,\dots,f_m)^*F'_m ) = \sd\cat_1(F'_m\circ(f_1,\dots,f_m)).
\]
This completes the proof.
\end{proof}

\begin{remark}
    Suppose we define a different map $F''_m\colon G^m\to G^{m-1}$ as follows:
    \[
    F''_m(x_1,\dots,x_{m-1},x_m)=(F'(x_1,x_m), F'(x_2,x_m), \dots, F'(x_{m-1},x_m)),
    \]
    where $F'$ is from Remark~\ref{rem: new map F}. Then again $\Ker((F''_m)_\#)=\Delta$, so our proof of Theorem~\ref{thm: main2} modifies to imply $\sd\D^{\cD}(f_1,\dots,f_m)=\sd\cat_1(F''_m\circ (f_1,\dots,f_m))$.
\end{remark}

Just like Theorem~\ref{thm: main1}, Theorem~\ref{thm: main2} recovers several results and computations from the literature. Since $\sd\cat_1(f)$ satisfies the same properties mentioned in Proposition~\ref{prop:catfprops} for $\sd\cat(f)$ (see~\cite{Op,OS,JO}), we can see the following.

\begin{enumerate}[(a)]
\item In the special case $X=G\times G$, we have $\sd\D^{\cD}(f_1,f_2)\le\sd\cat_1(G)$ due to Theorem~\ref{thm: main2} since $F_2'\circ (f_1,f_2)\colon G\times G\to G$.

\item It was shown in~\cite{PS,JO} that $\sd\TC_m^{\cD}(G)=\sd\cat_1(G^{m-1})$. This is recovered by the special case $X=G^m$ and $f_i=\pr_i\colon G^m\to G$ of Theorem~\ref{thm: main2} in view of Proposition~\ref{prop: recovering cd-tcf}.

\item If $G$ is group-like, then $\sd\D^{\cD}(\id_G,\eta)=\sd\D^{\cD}(\mu,\delta)$, where $\delta=\mu\circ(\eta\times\id_G)$. This follows by applying Theorem~\ref{thm: main2} and properties of $\sd\cat_1(f)$ to~\eqref{diag: antiparallel diagonal arrows}.

\item For $k\ge 2$, define $\mu_k\colon G\to G$ as the composition of $\mu\circ\Delta\colon G\to G$ with itself $(k-1)$-times, and set $\mu_0=c_e$ and $\mu_1=\id_G$. If $G$ is group-like, then $\sd\D^{\cD}(\mu_{k-\ell},c_e)=\sd\D^{\cD}(\mu_k,\mu_{\ell})$ by proceeding as in (iii) of Subsection~\ref{subsec: dist as cat}.

\item For a group-like space $G$ and an element $g\in G$, if $f_g\colon G \to G$ is defined as $f_g(h)=F(\mu(g,h),g)=\mu(\mu(g,h),\eta(g))$, then $\sd\D^{\cD}(f_g,c_{\eta(g)})=\sd\D(\id_G,c_e)$. We can just follow the steps from Example~\ref{exam:conjdist} and use the analogous results.
\end{enumerate}

\section{Maps to group-like spaces}
In this section, we focus specifically on maps from ANRs to group-like spaces. Our goal is to provide new upper bounds to the LS-category of such maps and to the sequential homotopic distance between them. We will do this using results from Subsection~\ref{subsec: dist as cat} and by discussing some Whitehead-type upper bounds given by the LS-category and sequential topological complexity of ANRs. Note that at the moment, we do not know the distributional analogues of these results because we only have an ``open set'' formulation in our current proofs. 

Recall that a group-like space $G$ is a homotopy-associative $H$-space with a compatible inversion map $\eta\colon G\to G$. We note that a path-connected CW homotopy-associative $H$-space is a group-like space, see~\cite[Sections~III.4 \& X.2]{Wh}. Indeed, if $\mu\colon G\times G\to G$ is the multiplication, then a shear map $\phi\colon G\times G\to G\times G$ given by $\phi(x,y)=(x,\mu(x,y))$ is a weak homotopy equivalence and hence a homotopy equivalence, from which it follows that $G$ is group-like.

\subsection{Whitehead-type theorems}\label{subsec: whitehead}

It is a classical result due to G.~Whitehead (see \cite[Chapter~X]{Wh})
that, for a group-like space $G$ and an ANR $X$, the homotopy classes of maps $X \to G$ form a nilpotent group $[X,G]$ of nilpotency class bounded above by $\cat(X)$. In this subsection, we revisit the proof of the above theorem and show its $\TC_m$-analogue for homotopy classes of some natural maps $X^m\to G$.

Let us recall some notions first. A topological group $H$ is \emph{nilpotent} if it has a central series
\[
H=\Gamma_1(H) \trianglerighteq \Gamma_2(H) \trianglerighteq \cdots \trianglerighteq \Gamma_c(H) \trianglerighteq \Gamma_{c+1}(H)=\{1\}.
\]
This means that $\Gamma_i(H)/\Gamma_{i+1}(H)$ is in the center of $H/\Gamma_{i+1}(H)$, which is equivalent to saying that $hkh^{-1}k^{-1} \in \Gamma_{i+1}(H)$ if $h\in H$ and $k \in \Gamma_i(H)$. The smallest $c$ is then the \emph{nilpotency class} of $H$, denoted $c(H)$. In other words, $c(H)$ is one less than the length of the smallest central series of $H$. 
Relative to a fixed central series as above, say we that $h\in H$ has \emph{depth} equal to $k$ if $h \in \Gamma_k(H)$, but $h \not\in \Gamma_j(H)$ for $j > k$. We shall denote this by $\dep_H(h)=k$ (while not denoting the particular central series). 

Since we will use a particular central series in our proofs, let's recall the proof of Whitehead's theorem. In fact, we will use the technique of this proof to obtain an analogue of Whitehead's theorem for sequential topological complexity.

\begin{theorem}\label{thm:whitehead}
If $G$ is a path-connected group-like space and $X$ is an ANR, then the set $[X,G]$ of homotopy classes of maps $X\to G$ is a nilpotent group of nilpotency class 
\[
c([X,G]) \leq \cat(X).
\]
\end{theorem}
\begin{proof}
First, we describe the group structure on $\cG = [X,G]$. The multiplication on $\cG$ is given by
\[
[\alpha] \cdot [\beta] = [\mu\circ (\alpha, \beta)],
\]
where $\mu\colon G \times G \to G$ is the multiplication on $G$, the inversion on $\cG$ is given by
\[
[\alpha]^{-1}=[\eta\circ \alpha].
\]
where $\eta\colon G \to G$ is the inversion on $G$, and the identity on $\cG$ is the constant map $c_e$ to the identity $e$ on $G$. This indeed gives a group structure on $\cG$ since $\mu$ is homotopy-associative and compatible with $\eta$. 

Let $\cat(X)=n$, so that $\{A_i\}_{i=0}^n$ is a closed \emph{categorial} cover of $X$, i.e., $\id_{|A_i}\simeq\ast$ for each $i$ (note that the ANR condition ensures that we may use such a \emph{closed} cover, see~\cite[Section~1.2]{CLOT}). For $1\le i \le n+1$, define the sets
\[
\cA_{i-1} = \bigcup_{k=0}^{i-1}A_k \quad \text{ and }\quad \Gamma_i(\cG) = \left\{[\alpha]\in [X,G]\,\middle|\, \alpha_{|\cA_{i-1}} \simeq *\right\}.
\]
Clearly, we have $\Gamma_1(\cG)=\cG$ since $\id_{|A_0}\simeq\ast$,
$\Gamma_{i+1}(\cG) \subset \Gamma_i(\cG)$ for $1\le i \le n$ since $\cA_{i-1} \subset \cA_{i}$, and $\Gamma_{n+1}(\cG)=\{c_e\}$ since $\cA_n=X$. 

Now, suppose $[\alpha]\in \cG$ and $[\beta]\in \Gamma_i(\cG)$ for some $i$. Because $X$ is an ANR, it has the homotopy extension property with respect to its closed subsets. Therefore, since $\alpha$ is null-homotopic on $A_i$, there exists a map $\wt \alpha_i\colon X \to G$ such that $\wt \alpha_i\simeq \alpha$ and $\wt \alpha_i(A_i)=e$ (note that the path-connectedness condition ensures the latter). Similarly, there exists a map $\wt \beta\colon X \to G$ such that $\wt\beta (\cA_{i-1})=e$ and $\wt\beta\simeq \beta$. In particular, $[\wt\beta]\in\Gamma_i(\cG)$. Next, we describe a commutator map in $\cG$. A 
commutator map $\Phi\colon G \times G \to G$ can be defined by 
$$\Phi = \mu \circ (\mu, \mu)\circ (\id_G, \id_G, \eta, \eta)\circ \Delta,$$
where $\Delta \colon G\times G \to (G\times G)^2$ is the diagonal map.
The key point to notice is that $\Phi|_{G \vee G}$ is null-homotopic, where $G\vee G$ has wedge-point $e$. For instance,
$$\Phi(g,e) = \mu(\mu(g,e),\mu(\eta(g),\eta(e)))\simeq \mu(g,\mu(\eta(g),e))\simeq \mu(g,\eta(g))\simeq c_e(g)=e$$
for each $g\in G$. The induced commutator $\Phi_{\cG}\colon\cG\times\cG\to\cG$ on $\cG$ is then given by
$$\Phi_{\cG}([\gamma],[\delta]) = [\Phi \circ (\gamma,\delta)].$$
For brevity, we write $\Phi_{\cG}([\gamma],[\delta])
=[\gamma,\delta]$. To justify that $\{\Gamma_j(\cG)\}_{j=1}^{n+1}$ gives a central series of $\cG$, we show that $[\alpha,\beta]\in \Gamma_{i+1}(\cG)$ for $[\alpha]\in \cG$ and $[\beta]\in\Gamma_i(\cG)$. By what has been
said above, we can replace $\alpha$ and $\beta$ with $\wt \alpha_i$ and $\wt\beta$, respectively. 
Consider the commutator $[\wt \alpha_i,\wt \beta] = \Phi \circ (\wt \alpha_i, \wt \beta)$. If we
restrict to $\cA_i$, then $\wt\alpha_i(A_i)=e$ and $\wt\beta(\cA_{i-1})=e$. Therefore, $(\wt\alpha_i,\wt\beta)(\cA_i)\subset G\vee G$, so that $\Phi \circ (\wt \alpha_i, \wt \beta)$ is null-homotopic on $\cA_i$. Hence, $\Phi_{\cG}([\alpha],[\beta])=[\alpha,\beta]=[\wt \alpha_i,\wt \beta]\in\Gamma_{i+1}(\cG)$. 
\end{proof}

Let us now present a (partial) $\TC_m$-analogue of Theorem~\ref{thm:whitehead}. As we will subsequently explain, this analogue is natural in a sense and applies in various interesting situations.

\begin{theorem}\label{thm: tc whitehead}
If $m\ge 2$ is an integer, $G$ is a path-connected group-like space, and $X$ is an ANR, then the set $[X^m,G]':=\{f\in [X^m,G]\mid f_{|\Delta_m X}\simeq\ast \}$, where 
$\Delta_m\colon X\to X^m$ is the diagonal map,  is a nilpotent group of nilpotency class 
\[
c([X^m,G]') \leq \TC_m(X).
\]
\end{theorem}
\begin{proof}
    The group structure on $\cG'=[X^m,G]'$ is defined in the same way as in the proof of Theorem~\ref{thm:whitehead}. Indeed, such multiplication and the inversion maps on $\cG'$ are well-defined since, if $[\alpha],[\beta]\in\cG'$, then $\mu\circ (\alpha, \beta)_{|\Delta_mX}\simeq\ast$ and $\eta\circ \alpha_{|\Delta_mX}\simeq\ast$, which means 
    \[
    [\alpha] \cdot [\beta] = [\mu\circ (\alpha, \beta)]\in\cG'\ \text{ and }\ [\alpha]^{-1}=[\eta\circ \alpha]\in\cG'.
    \]
Let $\TC_m(X)=n$, so that there is a \emph{closed} cover $\{B_i\}_{i=0}^{n}$ of $X^m$ and partial sections $\sigma_i\colon B_i\to X^I$, for $0\le i \le n$, of the fibration $\pi^X_m\colon X^I\to X^m$ that evaluates each path at $t_j$ for $1\le j \le m$ (again, such a closed cover exists because $X$, and hence $X^m$, is an ANR). For each $j$, define $\psi_j\colon I\to I$ as $\psi_j(t)=t(t_m-t_j)+t_j$. For each $i \in\{0,\dots, n\}$, the map $H_i\colon B_i\times I \to X^m$ given by
\[
H_i(x_1,\ldots,x_m,s)=\left(\sigma_i(x_1,\ldots,x_m)(\psi_1(t)),\ldots,\sigma_i(x_1,\ldots,x_m)(\psi_m(t))\right)
\]
is a homotopy from the inclusion $B_i\hookrightarrow X^m$ to a map with values in the diagonal $\Delta_mX$. In particular, for any map $f\colon X^m\to G$ with $f_{|\Delta_mX}\simeq \ast$, the restriction $f_{|B_i}$ is null-homotopic for all $i$. For $1\le i \le n+1$, as before, we define the sets
\[
\cB_{i-1} = \bigcup_{k=0}^{i-1}B_k \quad \text{ and }\quad \Gamma_i(\cG') = \left\{[\alpha]\in [X^m,G]'\,\middle|\, \alpha_{|\cB_{i-1}} \simeq *\right\},
\]
so that $\cG'=\Gamma_1(\cG')\supset
\Gamma_{2}(\cG') \supset \cdots\supset\Gamma_n(\cG')\supset\Gamma_{n+1}(\cG')=\{c_e\}$, where the map $c_e\colon X^m\to G$ is the constant map to $e$.

Now, suppose $[\alpha]\in \cG'$ and $[\beta]\in \Gamma_i(\cG')$ for some $i$. As explained above, $\alpha$ is null-homotopic on $B_i$, and $\beta$ is null-homotopic on $\cB_{i-1}$ by assumption, so we get maps $\wt\alpha_i,\wt\beta\colon X^m\to G$ such that $\wt \alpha_i(B_i)=e$, $\wt\beta (\cB_{i-1})=e$, $\wt \alpha_i\simeq\alpha$, and $\wt\beta\simeq\beta$, using the facts that $X^m$ is an ANR and $G$ is path-connected. The above homotopy equivalences imply $\wt\alpha_{i|\Delta_mX}\simeq \ast$ and $\wt\beta_{i|\Delta_mX}\simeq \ast$, so that $[\wt\alpha_i],[\wt\beta]\in\cG'$ indeed.

Also, using the commutator $\Phi$ on $G$ described in the proof of Theorem~\ref{thm:whitehead}, we get an analogous commutator $\Phi_{\cG'}$ on $\cG'$ such that $\Phi_{\cG'}([\gamma],[\delta]) = [\Phi \circ (\gamma,\delta)]$. Again, it is easy to see that $\Phi \circ (\gamma,\delta)_{|\Delta_mX}\simeq\ast$, so that $\Phi_{\cG'}$ is well-defined. Then the same arguments as before give $(\wt\alpha_i,\wt\beta)(\cB_i)\subset G\vee G$, which implies that $\Phi \circ (\wt \alpha_i, \wt \beta)$ is null-homotopic on $\cB_i$, and consequently $\Phi_{\cG'}([\alpha],[\beta])\in\Gamma_{i+1}(\cG')$. This shows that $\{\Gamma_j(\cG')\}_{j=1}^{n+1}$ is a central series of $\cG'$, thereby completing the proof.
\end{proof}

The elements of $[X^m,G]'$ occur naturally, especially when $X$ is also group-like. For instance, with $m=2$ and $X=G$, the map $F=\mu\circ(\id_G\times\pr_2)\colon G\times G\to G$ is null-homotopic on the diagonal and hence belongs to $[X\times X,G]'$.

\subsection{Applications to LS-category and homotopic distance}\label{subsec: applications}
First, we use Whitehead-type theorems to bound the LS-category of some maps from above.

The following result does not seem to have been noticed before. (Throughout, we follow the notations from the proof of Theorem~\ref{thm:whitehead}.)

\begin{proposition}\label{prop:catfbound}
If $G$ is a group-like space, $X$ is an ANR, and $f\colon X\to G$, then 
$$\cat(f) \leq \cat(X) - \dep_{\cG}(f) +1,$$
where $\dep_{\cG}(f)$ is the depth of $f$ in $\cG=[X,G]$, relative to the central series $\{\Gamma_j(\cG)\}_j$. In particular,
$\cat(X)=\dep_{\cG}(f)-1$ if and only if $f \simeq *$.
\end{proposition}
\begin{proof}
Let $\cat(X)=n$ with a closed categorical cover $\{A_i\}_{i=0}^{n}$. 
Also, suppose the (relative) depth of $f$ in $\cG$ is $\dep_{\cG}(f)=k$. Then $f\in \Gamma_k(\cG)$
and $f \not\in \Gamma_m(\cG)$ for $m > k$, so that $f|_{\cA_{k-1}}
\simeq *$. Because the cover is categorical, we also see that $f|_{A_j}
\simeq *$ for $j\in\{ k, \ldots, n\}$. Therefore, $f$ is null-homotopic on $1 + (n-k)+1$ sets that cover $X$. Hence,
$$
\cat(f) \leq n-k+1 = \cat(X) - \dep_{\cG}(f) + 1.
$$
For the last part, we know that $f \simeq *$ if and only if $\dep_{\cG}(f)=n+1$; that is
$f$ is null-homotopic on $\cA_n = X$. But $n=\cat(X)$, so the conclusion follows.
\end{proof}

We then get the following statement (in the notations of the proof of Theorem~\ref{thm: tc whitehead}), whose proof is entirely analogous to that of Proposition~\ref{prop:catfbound} and hence omitted.

\begin{proposition}\label{prop:catfbound-usingtcm}
If $G$ is a group-like space, $X$ is an ANR, and $g\colon X^m\to G$ is such that $g_{|\Delta_m X}\simeq\ast$, then 
$$\cat(g) \leq \TC_m(X) - \dep_{\cG'}(g) +1,$$
where $\dep_{\cG'}(g)$ is the depth of $g$ in $\cG'=[X^m,G]'$, relative to the central series $\{\Gamma_j(\cG')\}_j$. In particular, $\TC_m(X)=\dep_{\cG'}(g)-1$ if and only if $g \simeq *$.
\end{proposition}

\begin{remark}
Since $\TC_m(X)\le\cat(X^m)$ in general (with strict inequality holding for several spaces, notably for path-connected group-like spaces and CW $H$-spaces --- see Section~\ref{sec: h-spaces}), the upper bounds implied by Theorem~\ref{thm: tc whitehead} and Proposition~\ref{prop:catfbound-usingtcm} are strict improvements to the bounds implied respectively by Theorem~\ref{thm:whitehead} and Proposition~\ref{prop:catfbound} for maps $f\colon X^m\to G$ that are null-homotopic on the diagonal $\Delta_mX$. For example, if $X$ is the product of $k$ positive-dimensional spheres, then 
    \[
    c([X^m,G])\le km\quad\text{and}\quad c([X^m,G]')\le km-(k-\ell),
    \]
    where $\ell$ is the number of even-dimensional spheres in the product.
\end{remark}

Before proceeding further, let us give a concrete instance where Proposition~\ref{prop:catfbound-usingtcm} (and hence Theorem~\ref{thm: tc whitehead}) are useful.

\begin{ex}\label{ex: depth of Fm}
    Let $G$ be group-like and consider the map $F_m\colon G^m\to G^{m-1}$ defined in~\eqref{eq: map Fm}. Proposition~\ref{prop:catfbound} implies that
    \[
    \cat(F_m)\le\cat(G^m)-\dep_{\cG^{m-1}}(F_m)+1,
    \]
    where $\dep_{\cG^{m-1}}(F_m)$ is the depth of $F_m$ in the nilpotent group $\cG^{m-1}:=[G^m,G^{m-1}]$, relative to the central series $\{\Gamma_j^{m-1}(\cG^{m-1})\}_j$. The latter is defined in precisely the same way as $\{\Gamma_j(\cG)\}_j$ in the proof of Theorem~\ref{thm:whitehead}, which makes sense since $G^{m-1}$ is also group-like. By the definition of $F_m$ and the discussion in Remark~\ref{rem: map F}, it follows that $F_{m|\Delta_mG}\simeq \ast$. So, Proposition~\ref{prop:catfbound-usingtcm} applies and because $\TC_m(G)=\cat(G^{m-1})$, we get the upper bound
    \[
    \cat(F_m)\le\cat(G^{m-1})-\dep_{(\cG^{m-1})'}(F_m)+1,
    \]
    where $\dep_{(\cG^{m-1})'}(F_m)$ is the depth of $F_m$ in $(\cG^{m-1})':=[G^m,G^{m-1}]'$, relative to the central series $\{\Gamma_j^{m-1}((\cG^{m-1})')\}_j$, defined as in the proof of Theorem~\ref{thm: tc whitehead}. In particular, this upper bound is better than the previous one in the case when $F_m$ has the same depths in $\cG^{m-1}$ and $(\cG^{m-1})'$ relative to the above central series.
\end{ex}

Finally, we see an application of Proposition~\ref{prop:catfbound} to the sequential homotopic distance between maps from ANRs to group-like spaces.

\begin{proposition}\label{prop:catfbound}
    If $X$ is an ANR, $G$ is a group-like space, and $f_i\colon X\to G$ are maps for $1\le i \le m$, then we have that
    \[
    \D(f_1,\dots,f_m)\le\cat(X)-\min\left\{\dep_{\cG}(f_i)\ \middle|\ 1\le i \le m\right\}
    +1,
    \]
    where $\dep_{\cG}(f_i)$ is the depth of $f_i$ in $\cG=[X,G]$ for each $i$, relative to the central series $\{\Gamma_j(\cG)\}_j$.
\end{proposition}

\begin{proof}
Following the notations from Example~\ref{ex: depth of Fm}, we will write $\dep_{\cG^{m-1}}(f)$ for the depth of $f$ in $\cG^{m-1}$ relative to the aforementioned central series of $\cG^{m-1}$. Let 
\[
k=\min\left\{\dep_{\cG}(f_i)\ \middle|\ 1\le i \le m\right\},
\]
so that $f_{i|\A_{k-1}}\simeq\ast$. Hence, $\dep_{\cG^m}(f_1,\ldots,f_m) = k$ since 
$(f_1,\ldots,f_m)_{\A_{k-1}} \simeq \ast$ and $k$ is the minimum. 
Recall that the map $F_m\colon G^m\to G^{m-1}$ defined in~\eqref{eq: map Fm} satisfies
\[
F_m\circ (f_1,\dots,f_m)(x)=( \mu(f_1(x),\eta(f_2(x))), \dots, \mu(f_{m-1}(x),\eta(f_m(x))) ).
\]
Hence, $F_m\circ (f_1,\dots,f_m)_{|\A_{k-1}}\simeq\ast$, and so, $\dep_{\cG^{m-1}}(F_m\circ (f_1,\dots,f_m))\ge k$. But $\D(f_1,\dots,f_m)=\cat(F_m\circ(f_1,\dots,f_m))$ by Theorem~\ref{thm: main1} and 
\[
\cat(F_m\circ(f_1,\dots,f_m))\le \cat(X)-\dep_{\cG^{m-1}}(F_m\circ (f_1,\dots,f_m))+1
\]
by Proposition~\ref{prop:catfbound}. But, the map $F_m$ induces a map $\Gamma_i(\cG^m) \to 
\Gamma_i(\cG^{m-1})$, so that 
\[
\dep_{\cG^{m-1}}(F_m \circ (f_1,\ldots,f_m)) \geq 
\dep_{\cG^m}(f_1,\ldots,f_m)=k. 
\]
Hence, we may replace $\dep_{\cG^{m-1}}(F_m 
\circ (f_1,\ldots,f_m))$ by the relatively more computable quantity $k$.
This completes the proof.
\end{proof}

We conclude this section by showing that the above bounds are sharp.

\begin{ex}\label{exam:sharpineq}
For $1\le i \le m$, let $\iota_i\colon G^{m-1}\to G^m$ be the inclusion 
\[
\iota_i(x_1,\dots,x_{m-1})=(x_1,\dots,x_{i-1},x_0,x_{i},\dots,x_{m-1}).
\]
As explained in Proposition~\ref{prop: recovering cat}, $\D(\iota_1,\dots,\iota_m)=\cat(G^{m-1})$. Let $\pr_i\colon G^m\to G^{m-1}$ be the map $\pr_i(x_1,\dots,x_m)=(x_1,\dots,x_{i-1},x_{i+1},\dots,x_m)$, so that $\pr_i\circ\iota_i=\id_{G^{m-1}}$ for each $i$. Then $\cat(\iota_i)=\cat(G^{m-1})$ for each $i$, because Proposition~\ref{prop:catfprops} implies
$$\cat(G^{m-1}) = \cat(\id_{G^{m-1}}) \leq \cat(\iota_i) \leq \cat(G^{m-1}).
$$
Now, let $\cat(G^{m-1})=n$ and take a closed categorical cover $A_0, \ldots,A_n$. For some fixed $i$, if $\iota_i \in \Gamma_j(\cG)$ for $j > 1$ (where $\cG =
[G^{m-1},G^m]$), then $\iota_{i|\cA_{j-1}} \simeq *$. But then $\iota_i$
would be null-homotopic on the $n-j+2$ sets $\cA_{j-1}, A_j,\ldots, A_n$
and hence $\cat(\iota_i) < n=\cat(G^{m-1})$, which is a contradiction. Therefore, $\dep_{\cG}(\iota_i)=1$ for all $i$ and so, $\min\{\dep_{\cG}(\iota_i)\mid 1\le i \le m\}=1$. Hence,
\[
\cat(G^{m-1}) = \D(\iota_1,\dots,\iota_m) \leq \cat(G^{m-1}) - 1+1 =\cat(G^{m-1}).
\]
So, we have a case where the bounds in Propositions~\ref{prop:catfbound} and~\ref{prop:catfbound} are sharp.
\end{ex}

\section{Looping a map}\label{sec:looping}
What happens to homotopic distances when maps are looped? We will see in this section that (distributional) homotopic distance can go up or down. We need the following result.

\begin{lemma}\label{lem:zerohomo}
Suppose $f\colon X \to Y$ is a map of simply connected rational spaces that is 
zero on homotopy groups. Then $\Omega f\colon \Omega X
\to \Omega Y$ is null-homotopic.
\end{lemma}

For us, a \emph{rational space} is a topological space whose all homotopy groups are finite-dimensional rational vector spaces. For instance, the rationalization of a simply connected finite CW complex is a rational space. We shall consider these spaces in Section~\ref{sec:mapgroup} as well.

\begin{proof}
Adjoints to $1_{\Omega X}\colon \Omega X \to \Omega X$ and $\Omega f\colon 
\Omega X \to \Omega Y$ give a commutative diagram
$$\xymatrix{
\Sigma\Omega X \ar[r] \ar[dr]_-{f^\perp} & X \ar[d]^-f \\
& Y.
}
$$
Now, rationally, a suspension is a wedge of spheres and the homotopy class
of a map of a wedge of spheres to a space is determined by the classes restricted to the 
individual spheres. But these restrictions factor through $f$, which is null-homotopic
on all spheres (since it induces zero on homotopy groups). Hence, the map $f^\perp$
is null-homotopic, which implies that its adjoint $f$ is null-homotopic as well.
\end{proof}

This lemma can be used to analyze some examples of the effect of looping on (distributional) homotopic distance.

\begin{ex}\label{exam:Ddown}
Let $\delta\colon S^2 \times S^5 \to S^7$ be a degree $1$ map and $H\colon S^7 \to S^4$ be the Hopf fibration. Define $f:=H\circ \delta\colon S^2 \times S^5 \to S^4$. We use the same 
notation for the rationalizations of these
spaces and maps. Then, for degree reasons, $f_*=0$ on all (rational) homotopy groups. 
Note that $f$ cannot be null-homotopic
because, in that case, $\delta$ would factor through the fiber of $H$ which is $S^3$. 
But the inclusion of $S^3$ in $S^7$ \emph{is}
null-homotopic, so this would imply $\delta$ is null-homotopic and that contradicts the 
fact that its degree is $1$. Therefore, we
have $\sd\D(f,*)\not = 0$, while $\sd\D(\Omega f,*)=0$ in view of Lemma~\ref{lem:zerohomo}.
\end{ex}

\begin{ex}\label{exam:Dup}
Let $u\colon S^{2n+1} \to K(\Q,2n+1)$ for $n \geq 1$ be the non-zero rational cohomology 
class in $H^{2n+1}(S^{2n+1};\Q)$. Here, we use the general 
identification of $H^m(X;\Q)$ with homotopy classes $[X,K(\Q,m)]$ via $\beta \mapsto u$
by pulling back a chosen generator $\beta\in H^{m}(K(\Q,m);\Q)\cong\Q$. Then 
$\sd\cat(u)=1$, because it is not null-homotopic, and $\sd\cat(S^{2n+1})=1$. 
Thus, $\sd\D(u,*)=\sd\cat(u)=1$ by 
Proposition~\ref{prop: recovering tcf}. The rationalization of the sphere 
obeys $S^{2n+1}_\Q \simeq K(\Q,2n+1)$, so we have $\Omega S^{2n+1}_\Q \simeq 
\Omega K(\Q,2n+1) \simeq K(\Q,2n)$, so that 
$\Omega u\colon \Omega S^{2n+1} \to K(\Q,2n)$ is non-zero as well. Because $|\Omega u|=2n$ and $H^*(K(\Q,2n)$ is a polynomial algebra on $\Omega u$, 
we see that $u^k\not = 0$ for all $k>0$. Hence, $\sd\D(\Omega u,*)=\sd\cat(\Omega u)$ is infinite. 
\end{ex}

\begin{proposition}\label{prop:equalhomo}
If $f_i\colon X \to Y$ are based maps between simply connected rational spaces such that $Y$ is a CW $H$-space and $(f_i)_\#=(f_{i+1})_\#\colon \pi_j(X) \to \pi_j(Y)$ for all $j>1$ and $1\le i \le m-1$
then
\[
\sd\D(\Omega f_1, \dots,\Omega f_m)=0.
\]
That is, $\Omega f_i \simeq \Omega f_j$ for $1\le i,j\le m$.
\end{proposition}

\begin{proof}
Let $(Y,\mu,e)$ be an $H$-space structure and $F\colon Y\times Y\to Y$ be the map as in Remark~\ref{rem: map F} that produces a map $F_m\colon Y^m\to Y^{m-1}$ as in~\eqref{eq: map Fm}. Since $\Omega Y$ is group-like, it has a homotopy-associative multiplication $m$ with a compatible inversion $\iota$. Let $G=m\circ(\id_{\Omega Y}\times\iota)\colon\Omega Y\times\Omega Y\to \Omega Y$ and define $G_m\colon(\Omega Y)^m\to (\Omega Y)^{m-1}$ similar to $F_m$ but using $G$ instead of $F$ in the definition. Then Theorem~\ref{thm: main1} gives 
\[
\sd\D(\Omega f_1, \dots,\Omega f_m)=\sd\cat(G_m\circ (\Omega f_1, \dots,\Omega f_m)).
\]
There is another mutliplication $\Omega\mu$ on $\Omega Y$, where $\Omega\mu(\ell,\ell')(t)=\mu(\ell(t),\ell'(t))$, and similarly another inversion $\Omega\eta$ on $\Omega Y$. However, it is known that there is a unique multiplication in $[\Omega Y \times \Omega Y,\Omega Y]$ (see~\cite[Theorem 1.5]{Hil}), so these multiplications (and thus inversions) are  homotopic. Thus, for any $\ell\in\Omega X$ and $i\in\{1,\dots,m-1\}$,
\begin{align*}
    m\left(\Omega f_i(\ell),\iota\left(\Omega f_{i+1}(\ell)\right)\right)(t) & \simeq \Omega\mu \left(\Omega f_i(\ell),\Omega\eta\left(\Omega f_{i+1}(\ell)\right)\right)(t)
    \\
    & = \mu\left(f_i(\ell(t)),\eta\left(f_{i+1}(\ell(t))\right)\right)
    \\
    & = \mu\circ \left(\id_{\Omega Y}\times\eta \right) \left(f_i,\eta(f_{i+1})\right)(\ell(t))
    \\
    & = \Omega\left(\mu\circ \left(\id_{\Omega Y}\times\eta \right)\left(f_i,\eta(f_{i+1})\right)\right)(\ell)(t).
\end{align*}
It follows that $G_m\circ (\Omega f_1, \dots,\Omega f_m)\simeq \Omega (F_m\circ (f_1, \dots,f_m))$. Note that for each $i$, $(f_i)_\#=(f_{i+1})_\#$ on all homotopy groups, and so
\[
(F_m\circ (f_1,\dots,f_m))_\#=((f_1)_\#-(f_2)_\#,\dots ,(f_{m-1})_\#-(f_m)_\#)=0
\]
on all homotopy groups. Hence, $\Omega (F_m\circ (f_1, \dots,f_m))\simeq\ast$ by Lemma~\ref{lem:zerohomo}, and thus, 
\[
\sd\D(\Omega f_1, \dots,\Omega f_m)=\sd\cat(\Omega (F_m\circ (f_1, \dots,f_m)))=0.
\]
\end{proof}

\section{Maps to rational groups}\label{sec:mapgroup}

Recall G.~Whitehead's theorem from Section~\ref{subsec: whitehead}, that if $G$ is a compact group-like space, then $[X,G]$ is a compact nilpotent topological group. In this section, we study the rationalization $[X,G]_{\Q}$ of this nilpotent group (see \cite{HMR}) and obtain the following computationally effective theorem for homotopic distances.

\begin{theorem}\label{thm:nilfgstar}
Suppose $G$ is a path-connected compact group-like space of finite type and $X$ is a 
simply connected finite CW complex. For maps $f_i\colon X \to G$, $1\le i \le m$, denote
the rationalizations by $f_{i\Q}\colon X_\Q \to G_\Q$. Then we have that
\[
\D(f_{1\Q},\dots,f_{m\Q})=\cu_{\Q}\left(\Im\left(
\Delta_{m-1}^*\circ \bigotimes_{i=1}^{m-1}\left(f_i^*-f_{i+1}^*\right)\right)
\right),
\]
where 
$f_i^*$ is the map induced in rational cohomology by $f_i$ and 
$\Delta_{m-1}\colon X\to X^{m-1}$ is the diagonal map. If, additionally, $G$ has the 
homotopy type of a CW complex, then $\dD(f_{1\Q},\dots,f_{m\Q})=\D(f_{1\Q},\dots,f_{m\Q})$.
\end{theorem}

As noted previously in Section~\ref{sec:cohombounds}, the above image may not form a ring, but its rational cup length $\cu_\Q$ is still defined. Note also that the above exact value of $\sd\D$ in the case $m=2$ coincides with the general lower bound to $\sd\D$ between such maps by taking $c_1=-c_2$, see Remark~\ref{rem: dist lower bound} and Corollary~\ref{cor: same lower bound}. The latter gives further evidence of the sharpness of those bounds.


Let us begin by recalling that the rationalization of a compact path-connected topological 
group $G$ of finite type is a finite product of rationalized odd-dimensional spheres (see, for instance,~\cite[Section~9.1]{MP}), each of which is a rational Eilenberg--Mac~Lane space. 
Therefore, when $X$ is finite, the rationalization of $[X,G]$ is
\begin{equation}\label{eq:hclassQ}
[X,G]_\Q = \left[X,G_\Q\right] = \left[X,\prod_{k=1}^p S^{n_k}_\Q\right] 
=\left[X,\prod_{k=1}^p K(\Q,n_k)\right] 
= \prod_{k=1}^p \left[X,K(\Q,n_k)\right].
\end{equation}

\begin{remark}\label{rem: eilenberg-maclane}
Note that $f_{i\Q}\in [X,G]_\Q$ are determined (up to homotopy) by 
$(f_{i\Q,1},\ldots, f_{i\Q,p})$ for each $i$, 
where $f_{i\Q,k} \in H^{n_k}(X;\Q)$ for $1\le k\le p$ using 
the identification of the rational cohomology of $X$ in degree $n_k$ with homotopy 
classes $[X,K(\Q,n_k)]$ via pulling back a chosen generator $\beta_k$ of
$H^{n_k}(K(\Q,n_k);\Q)\cong\Q$ by $f_{i\Q,k}$. Of course, this means that if 
$f_{i\Q}^*=f_{i+1\Q}^*$ for each $i$, then $f_{i\Q}\simeq f_{i+1\Q}$, so that 
$\sd\D(f_{1\Q},\dots,f_{m\Q})=0$. The import of Theorem~\ref{thm:nilfgstar} then is to
extend this type of cohomology classification to the general case when 
$f^*_{i\Q}$ and $f^*_{i+1\Q}$ are not the same globally.     
\end{remark}

We need the following lemma first. Its proof uses tools from rational homotopy theory, so we postpone it to the appendix at the end of this paper.

\begin{lemma}\label{lem:catf}
Suppose $X$ is a simply connected finite CW complex and $G$ is a compact group-like 
space. Let $g \in [X,G]$, and let its rationalization $g_\Q \colon X_\Q \to G_\Q$ be represented 
as above by $g_\Q=(g_1,\ldots,g_p)$, where $G\simeq_\Q \prod_{k=1}^p 
K(\Q,n_k)$ and 
$g_k = g_\Q^*(\beta_k) \in H^{n_k}(X;\Q)$ for a generator 
$\beta_k \in H^{n_k}(K(\Q,n_k);\Q)$. Then
\[
\cat(g_\Q)=\cu_\Q(\textup{Im}(g^*))=\cu_\Q\{g_1,\ldots,g_p\}, 
\]
where $\cu_\Q\{g_1,\ldots,g_p\}$ is the length of the longest non-zero cup product(s) of the form $g_{j_1}\smile \cdots\smile g_{j_t}\ne 0$ and $j_k\in\{1,\dots,p\}$ for $1\le k \le t$.  If, additionally, $G$ has the homotopy type of a CW complex, then $\dcat(g_\Q)=\cu_\Q\{g_1,\ldots,g_p\}$ as well.
\end{lemma}

\begin{proof}[Proof of Theorem~\ref{thm:nilfgstar}]
For convenience, we write $f_i$ instead of the more unwieldy $f_{i\Q}$, but the
reader should always have the rationalized maps in mind.

First, we have the equality $\sd\D(f_1,\dots,f_m)=\sd\cat(F_m\circ (f_1,\dots,f_m))$ from Theorem~\ref{thm: main1}, 
where $F_m\colon G^m\to G^{m-1}$ is defined in~\eqref{eq: map Fm} using the composition $F=\mu\circ(\id_G\times\eta)\colon G\times G\to G$ for a multiplication $\mu$ and an 
inversion $\eta$ on $G$. Taking $F_m\circ (f_1,\dots,f_m)\in [X,G^{m-1}]_\Q$ in 
Lemma~\ref{lem:catf}, we deduce that
\[
\D(f_1,\dots,f_m)=\cu_\Q\left(\Im\left(F_m\circ (f_1,\dots,f_m)\right)^*\right).
\]
Indeed, $G^{m-1}$ is compact and group-like, so everything said in this section so far 
for $G$ applies to $G^{m-1}$ as well. 
Furthermore, if $G$ has the homotopy type of a CW complex, then $G^{m-1}$ has that of a 
finite CW complex, so that we also have
\[
\dD(f_1,\dots,f_m)=\cu_\Q\left(\Im\left(F_m\circ (f_1,\dots,f_m)\right)^*\right).
\]
For convenience, let us write $F_m\circ (f_1,\dots,f_m)=(f_1-f_2,f_2-f_3,\dots,f_{m-1}-f_m)$, where for each $i\in\{1,\dots,m-1\}$, the map $f_i-f_{i+1}\colon X_\Q\to G_\Q$ is the composition
\[
X _\Q \xrightarrow{\Delta}X _\Q\times X _\Q\xrightarrow{f_i\times f_{i+1}} G _\Q\times 
G _\Q\xrightarrow{1\times\eta} G _\Q\times G_\Q \xrightarrow{\mu} G _\Q,
\]
so that the map $F_m\circ (f_1,\dots,f_m)$
can also be seen simply as the composition
\begin{equation}\label{eq: writing Fm composition}
X _\Q \xrightarrow{\Delta_{m-1}} X^{m-1} _\Q \xrightarrow{\prod_{i=1}^{m-1}(f_i-f_{i+1})} 
G^{m-1} _\Q.
\end{equation}
Now, $H^*(G;\Q)\cong\bigwedge(\beta_{n_k})$, an exterior algebra on the generators 
of $H^{n_k}(K(\Q,n_k);\Q)$ for $1\le k\le p$, and it is a well-known fact that these generators are primitive, i.e.,
\[
\mu^*(\beta_{n_k}) = \beta_{n_k} \otimes 1 + 1 \otimes \beta_{n_k}.
\]
We see for the induced map $(f_i-f_{i+1})^*\colon H^*(G;\Q)\to H^*(X;\Q)$ that 
\begin{align*}
(f_i-f_{i+1})^*(\beta_{n_k}) & = \Delta^*((f_i^*\times f_{i+1}^*)((1\times \eta^*)(\mu^*(\beta_{n_k})))) \\
& = \Delta^*((f_i^*\times f_{i+1}^*)((1\times \eta^*)(\beta_{n_k} \otimes 1 + 1 \otimes \beta_{n_k}))) \\
& =  \Delta^*((f_i^*\times f_{i+1}^*)(\beta_{n_k} \otimes 1 - 1 \otimes \beta_{n_k})) \\
& = \Delta^*(f_i^*(\beta_{n_k}) \otimes 1 - 1 \otimes f_{i+1}^*(\beta_{n_k})) \\
& = f_i^*(\beta_{n_k}) - f_{i+1}^*(\beta_{n_k}),
\end{align*}
since $\Delta^*\colon H^*(X\times X;\Q)\to H^*(X;\Q)$ gives the cup product. 
Similarly, we see from~\eqref{eq: writing Fm composition} that $(F_m\circ (f_1,\dots,f_m))^*\colon H^*(G^{m-1};\Q)\cong H^*(G;\Q)^{\otimes m-1}\to H^*(X;\Q)$ on a generator $\beta_{1,n_k}\otimes\cdots\otimes \beta_{m-1,n_k}\in H^*(G;\Q)^{\otimes m-1}$ is
\begin{align*}
(F_m\circ (f_1,\dots,f_m))^*(\beta_{1,n_k}\otimes\cdots\otimes \beta_{m-1,n_k}) & = \Delta_{m-1}^*\left(\bigotimes_{i=1}^{m-1}\left(f_i^*\bigl(\beta_{i,n_k}\bigr) - f_{i+1}^*\bigl(\beta_{i,n_k}\bigr)\right)\right).
\end{align*}
Because the elements $\beta_{n_k}$ generate $H^*(G;\Q)$, and hence $\beta_{1,n_k}\otimes\cdots\otimes \beta_{m-1,n_k}$ generate $H^*(G^{m-1};\Q)$, we see that every element in $\Im(F_m\circ (f_1,\dots,f_{m}))^*$ is a product of elements as on the right side above. This completes the proof.
\end{proof}

We conclude with a few observations.

\begin{ex}
Let $G\simeq_\Q \prod_{k=1}^p K(\Q,n_k)$ and denote by $i_1\colon G_\Q \hookrightarrow G_\Q \times G_\Q$ the map $i_1(x)=(x,e)$. Then  $\sd\D(i_1,*)=\sd\cat(G_\Q)=k$, see Proposition~\ref{prop: recovering tcf}. This is verified by Theorem \ref{thm:nilfgstar} since $\mathrm{Im}(i_1^*)=H^*(G;\Q)$ and 
$\cu_\Q(H^*(G;\Q))=k$.
\end{ex}

\begin{ex}
Let $G\simeq_\Q \prod_{k=1}^p K(\Q,n_k)$ and denote by $\pr_i\colon G_\Q \times G_\Q \to G_\Q$ the respective projections. Then 
$\sd\D(\pr_1,\pr_2)=\sd\TC(G_\Q)=\sd\cat(G_\Q)=k$ by Proposition~\ref{prop: recovering tcf} and Theorem~\ref{thm: main1}. This is again verified by Theorem~\ref{thm:nilfgstar} since $(\pr_1^*-\pr_2^*)(x)=x\otimes 1 - 1\otimes x$ for all $x$, and the cup length of $\Im(\pr_1^*-\pr_2^*)$ (also known as the \emph{zero-divisor cup length} of $G_\Q$~\cite{Far}) in the case of a finite product of $k$ rationalized odd-dimensional spheres is $k$, so that we indeed have $\sd\D(\pr_1,\pr_2)=\cu_\Q(\Im(\pr_1^*-\pr_2^*))=k$. 
\end{ex}

\begin{ex}\label{exam:gauge}
If $\xi\colon G \to P \stackrel{p}{\to} X$ is a principal $G$-bundle (with $G$ a compact Lie group
and $X$ a finite CW complex, for instance), then the \emph{gauge group} of $\xi$, 
denoted $\mathcal G(\xi)$, is the 
group of $G$-equivariant homeomorphisms $h \colon P \to P$ satisfying $p\circ h = p$. It was 
shown in \cite{FO} that the rationalization of $\mathcal G(\xi)$ is the rationalized
mapping space $\text{Map}(X,G_\Q)$. So now, for two elements $h,k \in \mathcal G(\xi)$, 
take the rationalizations $h_\Q,k_\Q\colon X_\Q \to G_\Q$ and consider 
$$\D(h_\Q,k_\Q) = \cu_\Q(\Im(h^*-k^*))$$
as in Theorem \ref{thm:nilfgstar}. This gives a kind of \emph{homotopy metric} on the rationalized
gauge group $\mathcal G(\xi)_\Q$, in the sense that 
\begin{enumerate}
\item $\D(h_\Q,k_\Q)=0$ if and only if $h_\Q \simeq k_\Q$;
\item $\D(h_\Q,k_\Q) = \D(k_\Q,h_\Q)$;
\item $\D(h_\Q,k_\Q) \leq \D(h_\Q,g_\Q) + \D(g_\Q,k_\Q)$ for any $g \in \mathcal G(\xi)$.
This triangle inequality holds for any normal space $X$ (see \cite{MM,CD}). 
\end{enumerate}
Of course, Theorem \ref{thm:nilfgstar} says that this homotopy metric depends only
on the induced rational cohomology homomorphisms of the gauge group elements (thought
of as elements of $\text{Map}(X_\Q,G_\Q)$). As far as we are aware, this is a new 
invariant of gauge groups.
\end{ex}

\appendix

\section{Proof of Lemma \ref{lem:catf}}
A commutative differential graded algebra (cdga) is a \emph{Sullivan cdga} if its
underlying algebra is free commutative $\bigwedge V$ with $V=\{V^n\}$, $n \geq 1$
such that $V$ admits a well-ordered basis $x_\alpha$ having the property that
the differential obeys $d(x_\alpha) \in \bigwedge(x_\beta)_{\beta < \alpha}$. We write
$(\bigwedge V, d)$. The cdga $(\bigwedge V,d)$ is a \emph{Sullivan minimal cdga} if $d$ also
satisfies $d(V) \subset \bigwedge^{\geq 2} V$, where the superscript indicates 
elements whose constituent monomials are non-trivial products of elements of $V$.
Corresponding to each simply connected space $X$, there is a Sullivan minimal cdga
$(\bigwedge V,d)$ that gives the homotopy type of the rationalization $X_\Q$. 
Sullivan minimal models are unique up to isomorphism, and a cdga morphism
of Sullivan minimal cdgas $(\bigwedge V,d_V) \to (\bigwedge W,d_W)$ that induces isomorphisms on cohomology (i.e., a quasi-isomorphism) is, in fact, an isomorphism. 
If $f\colon (A,d_A) \to (B,d_B)$ is a cdga morphism, then there is a commutative diagram
$$\xymatrix{
A \ar[r]^-f \ar[dr]_-i & B \\
& (A \otimes \bigwedge V,d), \ar[u]_-\phi 
}
$$
where $i(a)=a$, $d|_A=d_A$, $d(V) \subset (A^+ \otimes \bigwedge V) \oplus \bigwedge^{\geq 2}V$
such that $V$ admits a well-ordered basis $(x_\alpha)$ with $d(x_\alpha) \in
A \otimes (\bigwedge (x_\beta))_{\beta < \alpha}$, and $\phi$ is a quasi-isomorphism. Here, $A^+$ simply denotes that we exclude all elements from $A^0\cong\Q$, and the map $i$ is called a \emph{relative minimal model} of $f$. If $h\colon X \to Y$ is a map of
simply connected spaces, then there is a (unique up to quasi-isomorphism) relative
minimal model $i\colon (\bigwedge V,d) \to (\bigwedge V \otimes \bigwedge W,d)$ of $h$ satisfying 
$h^*=i^*$, and a quasi-isomorphism
$(\bigwedge V \otimes \bigwedge W,d) \to (\bigwedge Z,d)$, where $(\bigwedge Z,d)$ and $(\bigwedge V,d)$
are the respective minimal models of $X$ and $Y$. Some basic references for rational homotopy 
theory are~\cite{FHT,FOT}.

With all this in mind, we say that for a map $h\colon X \to Y$,
the \emph{rational category} of $h$, denoted $\cat_0(h)$, is the least $n$
such that there is a cdga morphism $\rho$ with $\rho\circ j = i$ in the following
diagram:
$$\xymatrix{
\bigwedge V \ar[r]^-i \ar[d]_-{\text{pr}} \ar[dr]_-j & \bigwedge V \otimes \bigwedge W \\
\bigwedge V/\bigwedge^{>n}V & \bigwedge V \otimes \bigwedge Z \ar[l]_-\simeq \ar@{.>}[u]_-\rho .
}
$$
Here, $i$ is the relative model of $h$,   $\bigwedge V/\bigwedge^{>n} V$ is the cdga having all monomials of length greater
than $n$ killed, $\text{pr}$ is the quotient projection, and $j$ is the relative model of $\text{pr}$. In fact, we have
$\cat(h_\Q) = \cat_0(h)$ for maps of simply connected spaces, see \cite{Fel,FHT}.

In Lemma \ref{lem:catf}, we have $G_\Q \simeq \prod K(\Q,n_k)$ and the Sullivan 
minimal model is given by $(\bigwedge V,d=0)= \bigwedge(\beta_k)$, where 
$V^{n_k} \cong \oplus^t \Q$ such that
$t$ is the number of $K(\Q,n_k)$'s occurring in $G_\Q$. The vanishing of the 
differential means that every element of $\bigwedge V$ is a cocycle and therefore
gives a cohomology class. Now, by the usual cup length estimate, we know that
\[
\cu_\Q\left\{g_1,\dots,g_p\right\} \leq \cu_\Q(\text{Im}(g^*)) \leq \cat(g_\Q)=\cat_0(g),
\]
so we only need to show the opposite
inequality $\cat_0(g)\le \cu_\Q\{g_1,\dots,g_p\}$. Suppose that $\cat_0(g) = n < s$, where $s$ is the length of 
the longest non-zero product $g_{j_1}\smile \cdots\smile g_{j_s}$ in the notation
of Lemma~\ref{lem:catf} (i.e., $g_{j_t}=g_\Q^*(\beta_{j_t})$). The diagram defining 
$\cat_0(g)$ becomes here
$$\xymatrix@C=3pc{
& \bigwedge Z \\
\bigwedge (\beta_k) \ar[r]^-i \ar[d]_-{\text{pr}} \ar[dr]^-j \ar[ur]^-\psi  & 
\bigwedge (\beta_k) \otimes \bigwedge W \ar[u]_-{r}^-{\simeq}  \\
\bigwedge (\beta_k)/\bigwedge^{>n}(\beta_k) & \bigwedge (\beta_k) \otimes \bigwedge Z, 
\ar[l]^-{\simeq}_-{\phi} \ar[u]_-\rho 
}
$$
where $(\bigwedge Z,d)$ is the minimal model of $X$, $\psi\colon \bigwedge (\beta_k) \to \bigwedge Z$ 
is the morphism of minimal models corresponding to $g\colon X \to G$, and 
$i$ is the relative model of $\psi$. Now, $\psi^*=g^*=g_\Q^*$ in rational cohomology, so 
\[
\psi^*(\beta_{j_1}\cdots \beta_{j_s}) = g_{j_1}\smile \cdots\smile g_{j_s} 
\not = 0. 
\]
But then $i^*(\beta_{j_1}\cdots \beta_{j_s}) \not = 0$ since $i$ includes
the $\beta_k$'s into $\bigwedge (\beta_k) \otimes \bigwedge W$ and $r$ is a quasi-isomorphism.
Now, we have $\rho\circ j = i$, so $j^*(\beta_{j_1}\cdots \beta_{j_s}) \not = 0$ as
well, which then implies that $\text{pr}^*(\beta_{j_1}\cdots \beta_{j_s}) \not = 0$ since
$\phi$ is a quasi-isomorphism. But this is a contradiction because all words of
length greater than $n$ were killed and our assumption was $n < s$. Hence, 
$\cat_0(g)=n \geq s$ and we are done.

\end{document}